\documentclass[12pt]{amsart}

\pdfoutput=1

\usepackage[text={400pt,660pt},centering]{geometry}

\usepackage{esint,amssymb}
\usepackage{graphicx}
\usepackage{MnSymbol}
\usepackage{mathtools}
\usepackage[colorlinks=true, pdfstartview=FitV, linkcolor=blue, citecolor=blue, urlcolor=blue,pagebackref=false]{hyperref}
\usepackage{microtype}

\usepackage{bm}
\usepackage{dsfont}

\newtheorem{proposition}{Proposition}
\newtheorem{theorem}[proposition]{Theorem}
\newtheorem{lemma}[proposition]{Lemma}

\theoremstyle{remark}

\theoremstyle{definition}
\newtheorem{definition}[proposition]{Definition}

\numberwithin{equation}{section}
\numberwithin{proposition}{section}
\numberwithin{figure}{section}
\numberwithin{table}{section}

\newcommand{\N}{\mathbb{N}}

\newcommand{\R}{\mathbb{R}}

\newcommand{\E}{\mathbb{E}}
\renewcommand{\P}{\mathbb{P}}

\newcommand{\Rd}{{\mathbb{R}^d}}

\newcommand{\ep}{\varepsilon}
\newcommand{\eps}{\varepsilon}

\renewcommand{\le}{\leqslant}
\renewcommand{\ge}{\geqslant}

\renewcommand{\subset}{\subseteq}
\newcommand{\la}{\left\langle}
\newcommand{\ra}{\right\rangle}

\newcommand{\Ll}{\left}
\newcommand{\Rr}{\right}
\renewcommand{\d}{\mathrm{d}}

\renewcommand{\bar}{\overline}

\newcommand{\td}{\widetilde}

\newcommand{\1}{\mathds{1}}

\newcommand{\mcl}{\mathcal}

\newcommand{\al}{\alpha}

\newcommand{\de}{\delta}
\newcommand{\si}{\sigma}

\renewcommand{\t}{\mathsf{t}}

\begin{document}

\author[J.-C. Mourrat]{J.-C. Mourrat}
\address[J.-C. Mourrat]{DMA, Ecole normale sup\'erieure,
CNRS, PSL University, Paris, France}
\email{mourrat@dma.ens.fr}

\date{\today}

\title{Hamilton-Jacobi equations for mean-field disordered systems}

\begin{abstract}
We argue that Hamilton-Jacobi equations provide a convenient and intuitive approach for studying the large-scale behavior of mean-field disordered systems. This point of view is illustrated on the problem of inference of a rank-one matrix. 
We compute the large-scale limit of the free energy by showing that it satisfies an approximate Hamilton-Jacobi equation with asymptotically vanishing viscosity parameter and error term.
\end{abstract}

\maketitle

\section{Motivation}

The goal of this paper is to propose a new approach to the computation of the large-scale limit of the free energy of mean-field disordered systems. The new method is based on showing that the finite-volume free energy satisfies an approximate Hamilton-Jacobi equation, with viscosity parameter and error term that vanish in the large-scale limit. 

\smallskip

The paper grew out of my attempt to build some intuition for the celebrated Parisi formula for such systems, see \cite{MPV,guerra,Tpaper,Tbook2,pan}. The classical variational formulation of the free energy allows to write this quantity as the supremum of an energy and an entropy terms. For the Sherrington-Kirkpatrick model, the Parisi formula identifies the limit as an infimum instead, defying intuition. The plot thickens further when one considers more complicated systems such as the perceptron and the Hopfield models, which are expected to have limit free energies given by saddle-point variational problems \cite{obnoxious}.

\smallskip

In this paper, I propose a change of viewpoint that puts the main emphasis on the fact that the free energy satisfies a Hamilton-Jacobi equation, up to a small error. In this new point of view, this is the fundamental property that should be the center of attention and should receive an explanation. As is well-known, if the nonlinearity in the Hamilton-Jacobi equation is convex, then the solution can be expressed as an inf-sup variational problem. This suggests a Hamilton-Jacobi interpretation with convex nonlinearity for the Sherrington-Kirkpatrick model. However, it is unclear why one should expect the nonlinearity to always be convex. In fact, in the model that will be the focus of our attention here, the nonlinearity is concave, not convex. This still allows for a variational representation, but as a sup-inf instead of an inf-sup. More importantly, this suggests that the Hamilton-Jacobi point of view may be more robust and transparent than the variational representations.

\smallskip

The observation that finite-volume free energies satisfy approximate Hamilton-Jacobi equations already appeared in the physics literature \cite{barra1,barra2}. As was explained there, this idea can easily be made rigorous in the case of the Curie-Weiss model. Although the interactions in this model are not disordered, it is illustrative to explain the main ideas in this simple case.

\smallskip

For the Curie-Weiss problem, we would like to compute, for each $t \ge 0$, the large-$N$ limit of the free energy
\begin{equation*}  
\mathsf F_N^\circ(t) := \frac 1 N \log \sum_{\sigma \in \{\pm 1\}^N} 2^{-N} \exp \Ll( \frac t N \sum_{i,j = 1}^N \sigma_i \sigma_j \Rr).
\end{equation*}
We aim to do so by identifying a PDE satisfied by $\mathsf F_N^\circ(t)$, possibly up to error terms that vanish in the large-$N$ limit. However, at this stage we can only calculate derivatives with respect to $t$ and infer information about the distribution of $\sum_{i,j = 1}^N \sigma_i \sigma_j$ under the associated Gibbs measure. (For instance, the first and second derivatives are related to the mean and variance of this variable.) In order to find a closed set of equations, we need to ``enrich'' our free energy by introducing another quantity into the problem. This additional quantity should hopefully be simpler than $\sum_{i,j} \sigma_i \sigma_j$, and display some nontrivial correlations with the latter. In the present case, this quantity is very easy to guess: it is simply the average magnetization $\sum_{i = 1}^N \sigma_i$. (In more complicated settings, our intuition can be guided e.g.\ by cavity calculations.) Although we may a priori only care about calculating $\mathsf F_N^\circ(t)$, it is thus natural to introduce, for each $t \ge 0$ and $h \in \R$, the enriched free energy
\begin{equation*}  
\mathsf F_N(t,h) :=  \frac 1 N \log \sum_{\sigma \in \{\pm 1\}^N} 2^{-N} \exp \Ll( \frac t N \sum_{i,j = 1}^N \sigma_i \sigma_j + h \sum_{i = 1}^N \sigma_i\Rr).
\end{equation*}
Denoting by $\la \cdot \ra$ the associated Gibbs measure, we then observe that
\begin{equation*}  
\partial_t \mathsf F_N = \frac{1}{N^2} \la \sum_{i,j = 1}^N \sigma_i \sigma_j \ra \quad \text{ and } \quad 
\partial_h \mathsf F_N = \frac{1}{N} \la \sum_{i = 1}^N \sigma_i \ra,
\end{equation*}
so that
\begin{equation*}  
\partial_t\mathsf F_N - (\partial_h \mathsf F_N)^2 = \frac 1 {N^2} \la \Ll( \sum_{i = 1}^N \sigma_i - \sum_{i = 1}^N\la  \sigma_i \ra \Rr)^2 \ra.
\end{equation*}
Since the right side in the identity above is a variance, we should expect it to be small. Moreover, since $\mathsf F_N(t,h)$ encodes complete information on the law of $\sum \sigma_i$, it should be possible to find an expression for this variance in terms of~$\mathsf F_N$. We find indeed that
\begin{equation*}  
\partial_t\mathsf F_N - (\partial_h \mathsf F_N)^2 = \frac 1 N \partial_h^2 \mathsf F_N.
\end{equation*}
On this simple example, the free energy thus solves an exact Hamilton-Jacobi equation with viscosity term equal to $N^{-1}$. After observing that the value of $\mathsf F_N(0,h)$ does not depend on $N$, we have completely identified the limit $\mathsf F_\infty$ of~$\mathsf F_N$ as the viscosity solution to
\begin{equation*}  
\partial_t \mathsf F_\infty - (\partial_h \mathsf F_\infty)^2 = 0.
\end{equation*}
In a nutshell, due to the mean-field character of the model, we expect to be able to identify a handful of quantities whose statistics are related to one another. These relations will produce non-trivial identities between the first derivatives of the free energy: a Hamilton-Jacobi equation. There will be error terms, which one may expect to control by second-order derivatives, since these second-order derivatives are equal to the variances of the quantities of interest. 

\smallskip

We aim to carry an argument that has a similar structure for disordered mean-field models. However, for the Sherrington-Kirkpatrick and similar models, an important difficulty arises: the number of informative quantities one needs to add to the ``enriched'' free energy is infinite. In physicists' language, the system has a functional order parameter. As is well-known, this is bound to create very important technical difficulties. We will thus focus on the simpler setting provided by an inference problem. In this context, an additional symmetry forces the system to be replica-symmetric for every choice of parameters, and thus a simpler argument based on the addition of a single quantity suffices to ``close the equation''. We define the model on which we will focus and state our main results in the next section.

\section{Rank-one estimation, main results}
\label{s.results}

We consider the problem of estimating a vector $\bar x = (\bar x_1,\ldots, \bar x_N) \in \R^N$ of independent entries distributed according to a bounded measure $P$, given the observations of 
\begin{equation*}  
Y := \sqrt{\frac{t}{N}} \, \bar x \, \bar x^\t + W, 
\end{equation*}
where $W = (W_{ij})_{1 \le i,j \le N}$ are independent standard Gaussian random variables, independent of the vector $\bar x$. We denote the joint law of $\bar x$ and $W$ by $\P$, with associated expectation $\E$. Note that we seek to recover $N$ parameters from $N^2$ observations, each with a signal-to-noise ratio of the order of $N^{-1}$; this should therefore be the critical scaling for the inference of $\bar x$. 

\smallskip

By Bayes' rule, the posterior distribution of $\bar x$ given the observation of $Y$ is the probability measure 
\begin{equation}  
\label{e.def.posterior}
\frac { e^{H_N(t,x)} \, \d P_N(x)}{\int_{\R^N}e^{H_N(t,x')} \, \d P_N(x')} ,
\end{equation}
where we use the shorthand notation $P_N$ for the product measure $P^{\otimes N}$, and where $H_N(t,x)$ is defined by
\begin{align}  
\notag
H_N(t,x) & 
:=
\sqrt \frac t N \sum_{i,j = 1}^N Y_{ij} x_i x_j - \frac t {2N}\sum_{i,j =1}^N x_i^2 x_j^2
\\
\notag
& =  \sum_{i,j = 1}^N \Ll( \sqrt \frac t N W_{ij} x_i x_j + \frac t N x_i x_j \bar x_i \bar x_j- \frac t {2N} x_i^2 x_j^2 \Rr) 
\\
\label{e.def.H1}
& = \sqrt \frac t N x \cdot W x + \frac t N (x \cdot \bar x)^2 - \frac t {2N} |x|^4.
\end{align}
This will be explained in more details and in a slightly more general context in Appendix~\ref{s.nishi}. Note that although we suppress it from the notation, the quantity $H_N(t,x)$ is random in that it depends on the realization of $\bar x$ and $W$. Throughout the paper, we write $|x|$ to denote the $\ell^2$ norm of the vector $x \in \R^N$.

\smallskip

Our goal is to understand the large-$N$ behavior of the normalizing constant in \eqref{e.def.posterior}. The asymptotic behavior of this quantity has already been obtained multiple times in the literature; we refer to \cite{dm,lkz,bar1,lm,bm1,bmm,eak} for references. As was explained above, the point of the present paper is to devise yet another proof of this result, which centers on the identification of an appropriate Hamilton-Jacobi equation. Natually, several elements of the proof presented here can also be found in these previous works; the main difference is the global structure of the argument.


\smallskip 

In the spirit of the previous section, we start by introducing an ``enriched'' system. Let $z = (z_i)_{1 \le i \le N}$ be a vector of independent standard Gaussian random variables, independent of $\bar x$ and $W$ under $\P$. For every $t, {h} \ge 0$ and $x \in \R^N$, we define 
\begin{equation*}  
H_N(t,{h},x) := 
\sqrt \frac t N\,  x \cdot W x + \frac t N (x \cdot \bar x)^2 - \frac t {2N} |x|^4
+ \sqrt{h} \, z \cdot x + {h} \, x \cdot \bar x - \frac {h} 2 |x|^2.
\end{equation*}
The difference between the quantity above and that in \eqref{e.def.H1}, namely
\begin{equation*}  
 \sqrt{h} \, z \cdot x + {h} \, x \cdot \bar x - \frac {h} 2 |x|^2 = \sum_{i = 1}^N \Ll(\sqrt{{h}} z_i x_i + {h} x_i \bar x_i - \frac {h} 2 x_i^2\Rr),
\end{equation*}
is the energy associated with the much simpler inference prolem in which we try to recover $\bar x \in \R^N$ from the observation of $\sqrt {h} \, \bar x + z \in \R^N$.
We define the free energy
\begin{equation}  
\label{e.enriched.free.energy}
F_N(t,{h}) := \frac 1 N \log\Ll( \int_{\R^N} e^{ H_N(t,{h},x) } \, \d P_N(x) \Rr) ,
\end{equation}
as well as its expectation (with respect to the variables $\bar x$, $W$ and $z$)
\begin{equation}  
\label{e.expectation.free.energy}
\bar F_N(t,{h}) := \E \Ll[ F_N(t,h) \Rr] .
\end{equation}
For every $h \ge 0$, we set
\begin{equation}  
\label{e.def.psi}
\psi({h}) := \E \log \int_\R \exp\Ll(\sqrt{h} z_1 x +h x \bar x_1 - \frac h 2 x^2 \Rr) \, \d P(x)  = \bar F_1(0,h).
\end{equation}
In this expression, all the variables are scalar. Observe that $\bar F_N(0,h) = \psi(h)$ does not depend on $N$.
Our main goal is to prove the following result.

\begin{theorem}[Convergence to HJ]
\label{t.hj}
For every $M \ge 1$, we have
\begin{equation*}  
\lim_{N\to \infty} \E \Ll[ \sup_{[0,M]^2} (F_N - f)^2 \Rr] = 0,
\end{equation*}
where $f(t,h)$ is the viscosity solution of the Hamilton-Jacobi equation 
\begin{equation}  
\label{e.hj}
\Ll\{
\begin{aligned}
 \partial_t f - 2(\partial_{h} f)^2 & = 0 \quad && \text{in } (0,+\infty)^2, \\
 - \partial_h f & = 0 \quad && \text{on }  (0,+\infty) \times \{0\},
\end{aligned}
\Rr.
\end{equation}
with initial condition $f(0,h) = \psi(h)$. 
\end{theorem}
The next proposition is the main ingredient for the proof of Theorem~\ref{t.hj}. It states that the averaged free energy satisfies an approximate Hamilton-Jacobi equation with asymptotically vanishing viscosity parameter. 
\begin{proposition}[Approximate HJ in finite volume]
There exists $C < \infty$ such that for every $N \ge 1$ and uniformly over $[0,\infty)^2$,
\label{p.approx.hj}
\begin{equation*}  
0 \le \partial_t \bar F_N - 2(\partial_{h} \bar F_N)^2 \le \frac 2 {N}  \partial_{h}^2 \bar F_N  + 2 \E \Ll[ \Ll(\partial_h F_N - \partial_h \bar F_N\Rr)^2 \Rr]  +\frac{C}{N}\Ll(\frac 1 h + \frac 1 {\sqrt{h}}\Rr),
\end{equation*}
and moreover, 
\begin{equation}  
\label{e.boundary}
\partial_h \bar F_N\ge 0.
\end{equation}
\end{proposition}
In Proposition~\ref{p.approx.hj}, we kept the variables $(t,h)$ implicit for notational convenience. A more precise statement would be that for every $(t,h) \in [0,\infty)^2$, we have
\begin{equation*}  
0 \le \partial_t \bar F_N(t,h) - 2 \Ll( \partial_h\bar F_N(t,h) \Rr) ^2 \le \cdots
\end{equation*}
The right side of this inequality is interpreted as $+\infty$ when $h = 0$.
\smallskip

The next section is devoted to the proof of Proposition~\ref{p.approx.hj}. We will also give some basic estimates on the derivatives of $F_N$ and show that $F_N$ is concentrated around its expectation $\bar F_N$. Section~\ref{s.convergence} starts with the definitions relevant to the notion of viscosity solutions. We then prove Theorem~\ref{t.hj} using the results of Section~\ref{s.approx.hj}. The argument is similar to more standard situations for vanishing viscosity limits, although some additional difficulties appear. We close the section by discussing a variational representation for $f$ given by the Hopf-Lax formula. A generalization to tensors of arbitrary order is then obtained in Section~\ref{s.tensor}. In order to make the paper fully self-contained, two appendices are included. In Appendix~\ref{s.nishi}, we recall the proof of the Nishimori identity, which is a property of inference problems and is the main technical mechanism that allows to ``close the equation'' and remain in the replica-symmetric phase. In Appendix~\ref{s.visc}, we prove the comparison principle and the Hopf-Lax formula for viscosity solutions of~\eqref{e.hj}.

\section{Approximate Hamilton-Jacobi equation and basic estimates}
\label{s.approx.hj}
The main purpose of this section is to prove Proposition~\ref{p.approx.hj}. We will also record basic estimates on the derivatives of the free energy and its concentration properties that will be useful in the next section.

\smallskip

We denote by $\la \cdot \ra$ the Gibbs measure associated with the energy~$H_N(t,{h},\cdot)$. That is, for each bounded measurable function $f : \R^N \to \R$, we set
\begin{equation}  
\label{e.def.Gibbs}
\la f(x) \ra := \frac{1}{Z_N(t,{h})} \int_{\R^N} f(x) e^{H_N(t,{h},x)} \, \d P_N(x),
\end{equation}
where 
\begin{equation*}  
Z_N(t,{h}) := \int_{\R^N} e^{H_N(t,{h},x)} \, \d P_N(x).
\end{equation*}
Note that although the notation does not display it, this random probability measure depends on $t, {h}$, as well as on the realization of the random variables~$\bar x$, $W$ and $z$. 
We will also consider ``replicated'' (or tensorized) versions of this measure, and write $x$, $x'$, $x''$, etc.\ for the canonical ``replicated'' random variables. Conditionally on $\bar x$, $W$ and $z$, these random variables are independent and each is distributed according to the Gibbs measure $\la \cdot \ra$. Abusing notation slightly, we still denote this tensorized measure by $\la \cdot \ra$. An important ingredient for the proof of Proposition~\ref{p.approx.hj} is the Nishimori identity, which is a feature of inference problems whose proof is recalled in Appendix~\ref{s.nishi} below. For simplicity of notation, we only state this identity in the case of two or three replicas, since this will be sufficient for our purpose: for every bounded measurable function $f: \R^N \times \R^N \to \R$, we have
\begin{equation}  
\label{e.nishi}
\E \la f(x,x') \ra = \E \la f(x,\bar x) \ra,
\end{equation}
and for every bounded measurable function $f : \R^N \times \R^N \times \R^N \to \R$,
\begin{equation}  
\label{e.nishi2}
\E \la f(x,x',x'') \ra = \E \la f(x,x',\bar x)\ra.
\end{equation}

\begin{proof}[Proof of Proposition~\ref{p.approx.hj}]
We decompose the proof into three steps.

\smallskip

\emph{Step 1.} In this step, we compute the first derivatives of $\bar F_N$. Starting with the derivative with respect to $t$, we have
\begin{equation}  
\label{e.formula.partialt}
\partial_t  F_N(t,{h}) = \frac 1 N\la \frac {1} {2\sqrt {t N}} \,  x \cdot W x +  \frac 1 N (x\cdot \bar x)^2 - \frac 1 {2N} |x|^4 \ra.
\end{equation}
By Gaussian integration by parts, we have for every $i,j \in \{1,\ldots, N\}$ that
\begin{align*}  
\E\la W_{ij} x_i x_j \ra 
& = \E \Ll[ \partial_{W_{ij}} \la x_i x_j \ra \Rr] 
 = \sqrt \frac{t}{N} \, \E \la x_i^2 x_j^2 - x_i x_j x_i' x_j' \ra ,
\end{align*}
and thus, taking the expectation in \eqref{e.formula.partialt}, we get
\begin{equation*}  
\partial_t \bar F_N(t,{h}) = \frac {1}{2N^2} \E\la  -  (x \cdot x')^2 +  2 (x\cdot \bar x)^2\ra.
\end{equation*}
Using also the Nishimori identity \eqref{e.nishi}, we conclude that
\begin{equation}  
\label{e.formula.partialtbar}
\partial_t \bar F_N(t,{h}) = \frac 1 {2N^2} \E \la (x \cdot \bar x)^2 \ra.
\end{equation}
Similarly, since
\begin{align} 
\label{e.ibp.z}
\E \la z_i x_i \ra & = \E \Ll[ \partial_{z_i} \la x_i \ra \Rr] 
 = \sqrt{h} \, \E \la x_i^2 - x_i x_i' \ra ,
\end{align}
we have
\begin{align}  
\label{e.formula.partialmu1}
\partial_{h} \bar F_N(t,{h}) & = \frac 1 N \E \la \frac 1 {2\sqrt {h}} \, z \cdot x + x \cdot \bar x - \frac 1 2 |x|^2 \ra 
\\
\notag
& = \frac 1 N \E \la \frac {-1} 2 \, x \cdot x' + x \cdot \bar x \ra
\\
\label{e.formula.partialmu2}
& = \frac 1 {2N} \E \la x \cdot \bar x \ra.
\end{align}
We thus deduce that
\begin{equation}  
\label{e.variance.identity}
\partial_t \bar F_N - 2 (\partial_{h} \bar F_N)^2  = \frac 1 {2N^2} \E \la \Ll( x\cdot \bar x - \E\la x\cdot \bar x \ra \Rr)^2 \ra.
\end{equation}
In particular, this quantity is non-negative. 
Note also that
\begin{equation} 
 \label{e.positivity}
\partial_h \bar F_N(t,h) = \frac 1 {2N} \E \la x \cdot x' \ra = \frac{1}{2N} \E \Ll[ | \la x \ra|^2 \Rr] \ge 0,
\end{equation}
so that property \eqref{e.boundary} holds.

\smallskip

\emph{Step 2.} 
In the remaining two steps, we will control the right side of \eqref{e.variance.identity} in terms of the quantitites $\partial_h^2 \bar F_N$ and $\E \Ll[ (\partial_h F_N - \partial_h \bar F_N)^2 \Rr]$. In this step, we show that these quantities allow for a control of the fluctuations of 
\begin{equation*}  
H_N'(h,x) := \frac 1 {2\sqrt {h}} \, z \cdot x + x \cdot \bar x - \frac 1 2 |x|^2.
\end{equation*}
More precisely, we show that
\begin{multline}
\label{e.control.fluctH}
 \E \la \big(  H_N'(h,x) - \E \la H_N'(h,x) \ra \big) ^2 \ra 
\\
\le N \partial_h^2 \bar F_N(t,h) + N^2 \, \E \Ll[ \Ll( \partial_h F_N(t,h) - \partial_h \bar F_N(t,h) \Rr) ^2 \Rr] + C N h^{-1} .
\end{multline}
Our starting point is the variance decomposition
\begin{multline*}  
 \E \la \big(  H_N'(h,x) - \E \la H_N'(h,x) \ra \big) ^2 \ra 
\\
= \E \la \Ll( H_N'(h,x) - \la H_N'(h,x) \ra\Rr)^2 \ra + \E \Ll[\Ll( \la H_N'(h,x) \ra - \E \la H_N'(h,x) \ra\Rr)^2 \Rr] 
.
\end{multline*}
Since
\begin{equation}  
\label{e.formula.partialhF}
\partial_{h} F_N(t,{h}) = \frac 1 N \la H_N'({h},x) \ra ,
\end{equation}
and  $\bar F_N = \E [F_N]$, we readily have that
\begin{equation*}  
\E \Ll[\Ll( \la H_N'(h,x) \ra - \E \la H_N'(h,x) \ra\Rr)^2 \Rr] = N^2 \, \E \Ll[ \Ll( \partial_h F_N(t,h) - \partial_h \bar F_N(t,h) \Rr) ^2 \Rr].
\end{equation*}
We also have that
\begin{equation}
\label{e.partialh2F}
 \partial_{h}^2  F_N(t,{h})  = \frac{1}{N} \Ll(\la \Ll(H_N'({h},x)\Rr)^2 \ra -\la H_N'({h},x) \ra^2 \Rr) - \frac {1}{4N{h}^\frac 3 2} \la z \cdot x \ra ,
\end{equation}
and thus, taking expectations and using \eqref{e.ibp.z}, we get
\begin{equation*} 
\partial_{h}^2 \bar F_N(t,{h})
 = \frac{1}{N} \E \Ll[\la \Ll(H_N'({h},x)\Rr)^2 \ra -\la H_N'({h},x) \ra^2 \Rr] - \frac 1 {4N {h}} \E \la |x|^2 - x \cdot x' \ra.
\end{equation*}
Recall that we assume that the measure $P$ has bounded support. This implies that the last term in the display above is bounded by~$C h^{-1}$, and thus yields~\eqref{e.control.fluctH}.

\smallskip

\emph{Step 3.} In order to conclude, there remains to show that the variance of~$x \cdot \bar x$ is controlled by that of $H_N'(h,x)$. We show that
\begin{equation}
\label{e.relat.fluct}
  \E \la \Ll( x\cdot \bar x - \E\la x\cdot \bar x \ra \Rr)^2 \ra  \le  4 \E \la \big(  H_N'(h,x) - \E \la H_N'(h,x) \ra \big) ^2 \ra + \frac{CN}{\sqrt{h}}.
\end{equation}
In view of \eqref{e.formula.partialmu1} and \eqref{e.formula.partialmu2}, it suffices to show that
\begin{equation*}  
 \E \la \Ll( x\cdot \bar x \Rr)^2 \ra  \le  4 \E \la  H_N'(h,x) ^2 \ra  + \frac{CN}{\sqrt{h}}.
\end{equation*}
For every $i \neq j$, we have, using Gaussian integration by parts and the Nishimori identity,
\begin{align*}  
\E \la z_i z_j x_i x_j \ra 
& = \sqrt{h} \, \E \la z_j x_i x_j(x_i - x_i')\ra 
\\
& = h \, \E \la x_i x_j(x_i - x_i')(x_j + x_j' -2 \bar x_j) \ra,
\end{align*}
while for $i = j$,
\begin{align*}  
\E \la z_i^2 x_i^2\ra 
& = \sqrt{h} \, \E \Ll[ \la z_i x_i^2(x_i - x_i') \ra + \la x_i^2 \ra \Rr] 
\\
& = h \, \E \la x_i^2(x_i - x_i') (x_i + x_i' - 2 \bar x_i) \ra + \sqrt{h} \E \la x_i^2 \ra.
\end{align*}
As a consequence,
\begin{multline*}  
\E \la \Ll( \frac 1 {2\sqrt{h}} \, z \cdot x \Rr) ^2 \ra 
\\
= \frac 1 4 \Ll(\E \la |x|^4 - 2  |x|^2(x\cdot \bar x) -  (x\cdot \bar x)^2 + 2 (x\cdot \bar x)(x \cdot x')\ra+ \frac{1}{\sqrt{h}} \,  \E \la |x|^2 \ra\Rr).
\end{multline*}
Similarly,
\begin{multline*}  
\E \la \frac{2}{2\sqrt{h}} z\cdot x \Ll( x \cdot \bar x - \frac 1 2 |x|^2 \Rr) \ra 
\\
=   \E \la |x|^2 \Ll(  x \cdot \bar x - \frac 1 2 |x|^2 \Rr) \ra - \E \la (x\cdot x') \Ll(  x \cdot \bar x - \frac 1 2 |x|^2 \Rr) \ra.
\end{multline*}
We therefore obtain that 
\begin{align*}  
\notag
\E \la H_N'(h,x)^2 \ra & = 
\E \la \Ll( \frac 1 {2\sqrt{h}} \, z \cdot x \Rr) ^2 \ra 
+ 
\E \la \frac{2}{2\sqrt{h}} z\cdot x \Ll( x \cdot \bar x - \frac 1 2 |x|^2 \Rr) \ra 
\\
\notag
& \qquad
+
\E \la (x \cdot \bar x)^2 \ra +  \frac 1 4 \E \la |x|^4 \ra - \E \la (x\cdot \bar x)|x|^2 \ra
\\
& = \frac 3 4 \E \la (x\cdot \bar x)^2 \ra - \frac 1 2 \E \la (x \cdot \bar x)(x \cdot x') \ra + \frac 1 {4\sqrt{h}} \E \la |x|^2 \ra.
\end{align*}
By the Cauchy-Schwarz inequality and the Nishimori identity, we have 
\begin{equation*}  
\Ll| \E \la (x \cdot \bar x)(x \cdot x') \ra \Rr|  \le \E \la (x\cdot \bar x)^2 \ra,
\end{equation*}
and thus \eqref{e.relat.fluct} is proved. 
\end{proof}

Before turning to the proof of Theorem~\ref{t.hj}, we record simple derivative and concentration estimates in the next two lemmas. We use the notation
\begin{equation*}  
|W| := \sup \{|W x|  \ : \ x \in \R^N, \ |x|\le 1\}.
\end{equation*}
Of course, this quantity depends on $N$, and as we will see in the proof of Lemma~\ref{l.concentration}, it grows like~$\sqrt{N}$. The notation may be slightly misleading, in that it does not display the dependency on $N$. A similar convention is already in place when we write~$|x|$ to denote the $\ell^2$ norm of the vector $x \in \R^N$, a quantity which is typically of the order of $\sqrt{N}$.

\begin{lemma}[Derivative estimates]
\label{l.gradient}
There exists a constant $C < \infty$ such that the following estimates hold uniformly over $[0,\infty)^2$:
\begin{equation}  
\label{e.gradest.barF}
|\partial_t \bar F_N| + |\partial_h \bar F_N|  
\le C,
\end{equation}
\begin{equation}  
\label{e.gradest.Ft}
|\partial_t F_N| \le C + \frac{C|W|}{\sqrt{Nt}},
\end{equation}
\begin{equation}  
\label{e.gradest.F}
|\partial_h F_N| \le C + \frac{C|z|}{\sqrt{Nh}},
\end{equation}
\begin{equation}  
\label{e.conv.F}
\partial_h^2 F_N \ge - \frac{C|z|}{N^\frac 1 2 h^\frac 3 2}.
\end{equation}
\end{lemma}
\begin{proof}
Recall that the measure $P$ has bounded support. The estimates in  \eqref{e.gradest.barF}, \eqref{e.gradest.Ft}, \eqref{e.gradest.F}, and \eqref{e.conv.F} are thus immediate consequences of \eqref{e.formula.partialtbar}-\eqref{e.formula.partialmu2}, \eqref{e.formula.partialt}, \eqref{e.formula.partialhF}, and \eqref{e.partialh2F} respectively.
\end{proof}

We now turn to a concentration estimate. Since this is sufficient for our purposes, we simply state an $L^2$ bound in the probability space, and prove it using the elementary Efron-Stein inequality. The statement could be strengthened to a Gaussian-type integrability using concentration results such as \cite[Theorem~5.5 and Theorem~2.8]{blm} (and this also allows to improve the rate of decay to $0$ as $N$ tends to infinity).
\begin{lemma}[Concentration of free energy]
\label{l.concentration}
There exists $C < \infty$ such that for every $M \ge 1$ and $N \in \N$,
\begin{equation*}  
\E \Ll[ \sup_{[0,M]^2} \Ll( F_N - \bar F_N \Rr)^2  \Rr] \le C M^\frac 4 3 N^{-\frac 1 3}.
\end{equation*}
\end{lemma}
\begin{proof}
We recall that $\bar F_N$ is the expectation of $F_N$ with respect to the variables $\bar x$, $W$ and $z$. The Efron-Stein inequality gives us that
\begin{align*}  
\E \Ll[ \Ll( F_N - \bar F_N \Rr)^2  \Rr] 
& \le \sum_{1\le i,j \le N} \E \Ll[ \Ll(F_N - \E[F_N \ | \ W_{ij}]\Rr)^2 \Rr] 
\\
& \qquad + \sum_{1\le i \le N} \E \Ll[ \Ll(F_N - \E[F_N \ | \ z_{i}]\Rr)^2 \Rr] 
\\
& \qquad + \sum_{1\le i \le N} \E \Ll[ \Ll(F_N - \E[F_N \ | \ \bar x_{i}]\Rr)^2 \Rr] .
\end{align*}
By the Gaussian Poincar\'e inequality (see e.g.\ \cite[(2.5)]{chatterjee} or \cite{poincare}), we have 
\begin{equation*}  
\E \Ll[ \Ll(F_N - \E[F_N \ | \ W_{ij}]\Rr)^2 \Rr] \le \E \Ll[ \Ll( \partial_{W_{ij}} F_N \Rr) ^2 \Rr] ,
\end{equation*}
and 
\begin{equation*}  
\E \Ll[ \Ll(F_N - \E[F_N \ | \ z_{i}]\Rr)^2 \Rr] \le \E \Ll[ \Ll( \partial_{z_i} F_N \Rr) ^2 \Rr] .
\end{equation*}
Moreover,
\begin{equation*}  
\partial_{W_{ij}} F_N = t^\frac 1 2 \, N^{-\frac 3 2} \la x_i x_j \ra,
\end{equation*}
and
\begin{equation*}  
\partial_{z_i} F_N = h^\frac 1 2 \, N^{-1} \, \la x_i \ra,
\end{equation*}
so that 
\begin{equation*}  
\sum_{1\le i,j \le N} \E \Ll[ \Ll(F_N - \E[F_N \ | \ W_{ij}]\Rr)^2 \Rr]  + 
 \sum_{1\le i \le N} \E \Ll[ \Ll(F_N - \E[F_N \ | \ z_{i}]\Rr)^2 \Rr] 
 \le \frac{C (t+h)}{N}.
\end{equation*}
Since
\begin{equation*}  
\partial_{\bar x_i} F_N = \frac 1 N \la \frac {2t} N x_i \Ll( x \cdot \bar x \Rr)  + hx_i\ra  
\end{equation*}
is bounded by $C(t+h)/N$, and the support of the law of $\bar x_i$ is bounded, we also have that 
\begin{equation*}  
\sum_{1\le i \le N} \E \Ll[ \Ll(F_N - \E[F_N \ | \ \bar x_{i}] \Rr)^2 \Rr]\le \frac{C}N(t+h)^2.  
\end{equation*}
We have thus shown that there exists $C < \infty$ such that for every $(t,h) \in [0,\infty)^2$,
\begin{equation}  
\label{e.pointwise}
\E \Ll[ \Ll( F_N - \bar F_N \Rr)^2 (t,h) \Rr] \le \frac{C}{N}(t+t^2 + h+h^2).
\end{equation}
In order to complete the proof, there remains to use a regularity estimate for $F_N - \bar F_N$. By Lemma~\ref{l.gradient}, for every $t,t',h,h' \ge 0$ satisfying 
\begin{equation*}  
|t-t'| + |h-h'| \le 1,
\end{equation*}
we have
\begin{equation*}  
\Ll|F_N(t,h) - F_N(t',h')\Rr| \le C \Ll( 1  + \frac{|W|}{\sqrt{N}} +  \frac{|z|}{\sqrt{N}} \Rr) \Ll( |t-t'|^\frac 1 2 + |h-h'|^\frac 1 2 \Rr) .
\end{equation*}
On the other hand, it is clear from Lemma~\ref{l.gradient} that $\bar F_N$ is uniformly Lipschitz continuous, so in particular the estimate above also holds if $F_N$ is replaced by $\bar F_N$. Hence, for any $\eps \in (0,1]$, if we set
\begin{equation*}  
A_\ep := \ep \N^2 = \{0,\ep, 2\ep,\ldots\}^2,
\end{equation*}
then 
\begin{equation*}  
\sup_{[0,\infty)^2} |F_N - \bar F_N| - \sup_{A_\ep} |F_N - \bar F_N| \le C \Ll( 1  + \frac{|W|}{\sqrt{N}} +  \frac{|z|}{\sqrt{N}} \Rr) \sqrt{\ep}.
\end{equation*}
Moreover, for every $M \in [1,\infty)$, we have by \eqref{e.pointwise} that
\begin{align*}  
\E \Ll[ \sup_{A_\ep \cap [0,M]^2} (F_N - \bar F_N)^2\Rr] 
& \le \sum_{(t,h) \in A_\ep \cap [0,M]^2} \E \Ll[ (F_N - \bar F_N)^2(t,h) \Rr] 
\\
& \le C M^4 \ep^{-2} N^{-1}.
\end{align*}
Combining the two previous displays yields
\begin{equation*}  
\E \Ll[ \sup_{[0,M]^2} (F_N - \bar F_N)^2 \Rr] \le C \ep \E \Ll[  1 + \frac{|W|^2}{{N}} +  \frac{|z|^2}{{N}} \Rr] + \frac{C M^4}{N\ep^2},
\end{equation*}
and we clearly have $\E[|z|^2] = N$. In order to conclude, there remains to verify that $\E[|W|^2] \le C N$ (and then choose $\ep = M^\frac 4 3 N^{-\frac 13}$). For every fixed $x \in \R^N$ satisfying $|x| \le 1$ and every $i \in \{1,\ldots,N\}$, we have that $(Wx)_i$ is a centered Gaussian random variable with variance $|x|^2 \le 1$, and moreover, the random variables $((Wx)_i)_{1\le i \le N}$ are independent. We deduce that there exists $C < \infty$ such that for every $|x| \le 1$,
\begin{equation*}  
\E \Ll[ \exp \Ll( C^{-1} |Wx|^2 \Rr)  \Rr] \le C\exp \Ll( C N \Rr) ,
\end{equation*}
and thus by the Chebyshev inequality, after enlarging $C < \infty$ if necessary, we have that for every $a \ge C$,
\begin{equation*}  
\P \Ll[ |Wx|^2 \ge aN \Rr] \le \exp \Ll( -\frac{aN}{C} \Rr) .
\end{equation*}
Now, let $A \subset \R^N$ be a finite set such that any two points in $A$ are at distance at least $1/2$ from one another, and no point of $\{|x|\le 1\}$ can be added to $A$ without violating this property. By this property of maximality, it must be that for every $x$ satisfying $|x|\le 1$, there exists $y \in A$ such that $|x-y|\le 1/2$. Since for any $x,y \in \R^N$, we have
\begin{equation*}  
|Wx - Wy| \le |W| \, |x-y|,
\end{equation*}
it follows that
\begin{equation*}  
|W| = \sup_{|x|\le 1} |Wx| \le \sup_{x \in A} |Wx| + \frac{|W|}{2},
\end{equation*}
and thus
\begin{equation*}  
|W| \le 2 \sup_{x \in A} |Wx|.
\end{equation*}
In order to construct such a set $A$, we simply pick points in $\{|x|\le 1\}$ in some arbitrary manner, until the maximality property is reached. Note that the balls centered at each of the points in $A$ and of radius $1/4$ are disjoint; they are also contained in the ball of radius $5/4$. Computing the volume of these sets, we infer that $|A| \le 5^N$, and thus, by a union bound, we have for every $a \ge C$ that
\begin{equation*}  
\P \Ll[ |W|^2 \ge aN \Rr] \le 2 \exp \Ll( -N \Ll( \frac a C - \log 5 \Rr)  \Rr).
\end{equation*}
This implies in particular that $\E \Ll[|W|^2  \Rr] \le C N$, as desired. 
\end{proof}

\section{Convergence to viscosity solution}
\label{s.convergence}
The main goal of this section is to show that Proposition~\ref{p.approx.hj} implies Theorem~\ref{t.hj}. We also comment on variational representations for solutions of Hamilton-Jacobi equations at the end of the section. To start with, we recall the definition of viscosity solutions.

\begin{definition}  
\label{def.visc}
We say that a function $f \in C([0,\infty)^2)$ is a \emph{viscosity subsolution} of \eqref{e.hj} if for every $(t,h) \in (0,\infty)\times [0,\infty)$ and $\phi \in C^\infty((0,\infty)\times [0,\infty))$ such that $(t,h)$ is a local maximum of $f-\phi$, we have
\begin{equation*}
\Ll\{
\begin{aligned}  
\Ll(\partial_t \phi - 2(\partial_h \phi)^2\Rr)(t,h) \le 0 & \qquad \text{if } h > 0, \\
\min \Ll( -\partial_h \phi,\partial_t \phi - 2(\partial_h \phi)^2 \Rr)(t,h) \le 0 & \qquad \text{if } h = 0.
\end{aligned}
\Rr.
\end{equation*}
We say that a function $f \in C([0,\infty)^2)$ is a \emph{viscosity supersolution} of \eqref{e.hj} if for every $(t,h) \in (0,\infty)\times [0,\infty)$ and $\phi \in C^\infty((0,\infty)\times [0,\infty))$ such that $(t,h)$ is a local minimum of $f-\phi$, we have
\begin{equation*}
\Ll\{
\begin{aligned}  
\Ll(\partial_t \phi - 2(\partial_h \phi)^2\Rr)(t,h) \ge 0 & \qquad \text{if } h > 0, \\
\max \Ll( -\partial_h \phi , \partial_t \phi - 2(\partial_h \phi)^2\Rr)(t,h) \ge 0 & \qquad \text{if } h = 0.
\end{aligned}
\Rr.
\end{equation*}
We say that a function $f \in C([0,\infty)^2)$ is a \emph{viscosity solution} of \eqref{e.hj} if it is a viscosity sub- and supersolution. We may also say that a function $f \in C([0,\infty)^2)$ is a viscosity solution of 
\begin{equation}  
\label{e.hj.ineq}
\Ll\{
\begin{aligned}
 \partial_t f - 2(\partial_{h} f)^2 & \le 0 \quad && \text{in } (0,+\infty)^2, \\
 - \partial_h f & \le 0 \quad && \text{on }  (0,+\infty) \times \{0\},
\end{aligned}
\Rr.
\end{equation}
if it is a viscosity subsolution of \eqref{e.hj}. Similarly, we may say that a function $f \in C([0,\infty)^2)$ is a viscosity solution of~\eqref{e.hj.ineq} with the inequalities reversed if it is a viscosity supersolution of \eqref{e.hj}.
\end{definition}

The mechanism allowing to identify uniquely the viscosity solution to \eqref{e.hj} subject to appropriate initial condition relies on the following classical comparison principle. 
\begin{proposition}[Comparison principle]
\label{p.comparison}
Let $u$ be a subsolution and $v$ be a supersolution of \eqref{e.hj} such that both $u$ and $v$ are uniformly Lipschitz continuous in the variable $h$.
We have
\begin{equation*}  
\sup_{[0,\infty)^2} (u-v) = \sup_{\{0\}\times [0,\infty)} (u-v).
\end{equation*}
\end{proposition}
The proof of Proposition~\ref{p.comparison} is given in Appendix~\ref{s.visc}. (Besides the inconvenience that the domain under consideration is unbounded, the proof is classical.) In the statement of Proposition~\ref{p.comparison}, we assume a certain uniform Lipschitz continuity property in the variable $h$. As will be clear from the proof, this assumption can be weakened, and possibly be removed. This assumption is meant to allow for a simpler proof, and is not causing additional difficulties elsewhere since it is very easy to check that our candidate solutions satisfy it.

\smallskip

We are now ready to prove Theorem~\ref{t.hj}.

\begin{proof}[Proof of Theorem~\ref{t.hj}]
By Lemma~\ref{l.concentration}, it suffices to study the convergence of~$\bar F_N$ as $N$ tends to infinity.
Recall that $\bar F_N(h,0) = \psi(h)$ does not depend on $N$. Moreover, it is clear from \eqref{e.formula.partialt} and \eqref{e.formula.partialmu2} that $\bar F_N$ is uniformly Lipschitz in both variables. Hence, by the Arzel\'a-Ascoli theorem, the sequence $(\bar F_N)$ is precompact for the topology of local uniform convergence. Let $f$ be such that~$\bar F_N$ converges to $f$ locally uniformly as $N$ tends to infinity along a subsequence. For notational convenience, we will omit to refer to the particular subsequence along which this convergence holds. Our goal is to show that $f$ is a viscosity solution of \eqref{e.hj}. By the comparison principle (Proposition~\ref{p.comparison}), this would identify $f$ uniquely, and thus prove the theorem. 

\smallskip

We decompose the rest of the proof into six steps.

\smallskip

\emph{Step 1.} We show that $f$ is a viscosity supersolution of \eqref{e.hj}. It is easy to show that in the definition of viscosity supersolution, replacing the phrase ``local minimum'' by ``strict local minimum'' yields an equivalent definition. Let $(t,h) \in (0,\infty)\times [0,\infty)$ and $\phi \in C^\infty((0,\infty)\times [0,\infty))$ be such that $f-\phi$ has a strict local minimum at the point $(t,h)$. Since $\bar F_N$ converges to $f$ locally uniformly, there exists a sequence $(t_N,h_N) \in (0,\infty)\times [0,\infty)$ converging to $(t,h)$ as $N$ tends to infinity and such that $\bar F_N - \phi$ has a strict local minimum at $(t_N,h_N)$. If $h_N > 0$ infinitely often, then along a subsequence on which this property holds, we have that the first derivatives of $\bar F_N$ and $\phi$ at $(t_N,h_N)$ coincide, and thus by Proposition~\ref{p.approx.hj} that 
\begin{equation*}  
\Ll(\partial_t \phi - 2 (\partial_h \phi)^2 \Rr)(t_N,h_N) \ge 0.
\end{equation*}
By continuity, this implies that 
\begin{equation*}  
\Ll(\partial_t \phi - 2 (\partial_h \phi)^2 \Rr)(t,h) \ge 0,
\end{equation*}
as desired. There remains to consider the case when $h_N = 0$ infinitely often. In this case, we must have $h = 0$. We can also assert that 
\begin{equation}  
\label{e.compare.grad.bdry}
\partial_h (\bar F_N - \phi)(t_N,h_N) \ge 0, \qquad \partial_t (\bar F_N - \phi)(t_N,h_N) = 0 ,
\end{equation}
and we recall that, by Proposition~\ref{p.approx.hj},
\begin{equation}  
\label{e.recall.approx.hj}
\Ll(\partial_t \bar F_N - 2 (\partial_h \bar F_N)^2 \Rr)(t_N,h_N) \ge 0.
\end{equation}
If $-\partial_h\phi(t,h) \ge 0$, then there is nothing to show. Otherwise, using the first statement in \eqref{e.compare.grad.bdry}, we find that
\begin{equation*}  
(\partial_h \phi)^2(t,h) = \lim_{N \to \infty} (\partial_h \phi)^2(t_N,h_N) \le \liminf_{N \to \infty} (\partial_h \bar F_N)^2(t_N,h_N),
\end{equation*}
and thus, using also the second statement in \eqref{e.compare.grad.bdry} and \eqref{e.recall.approx.hj},
\begin{equation*}  
\Ll(\partial_t \phi - 2 (\partial_h \phi)^2 \Rr)(t,h) \ge \limsup_{N\to \infty} \Ll(\partial_t \bar F_N - 2 (\partial_h \bar F_N)^2 \Rr)(t_N,h_N) \ge 0.
\end{equation*}
This completes the proof of the fact that $f$ is a supersolution.

\smallskip

\emph{Step 2.} We next show that $f$ is a subsolution of \eqref{e.hj}. In this step, we focus on contact points of the form $(t,0)$; that is, we give ourselves $t > 0$ and $\phi \in C^\infty((0,\infty)\times [0,\infty))$ such that $f-\phi$ has a strict local maximum at the point $(t,0)$. In this case, there exists a sequence $(t_N,h_N) \in (0,\infty)\times [0,\infty)$ converging to $(t,0)$ and such that $\bar F_N - \phi$ has a local maximum at $(t_N,h_N)$. If $h_N = 0$, then we must have that 
\begin{equation*}  
\partial_h (\bar F_N - \phi)(t_N,h_N) \le 0.
\end{equation*}
This inequality still holds, and is in fact an equality, if $h_N > 0$. In view of \eqref{e.boundary}, we thus deduce that 
\begin{equation*}  
-\partial_h \phi(t_N,h_N) \le 0.
\end{equation*}
Letting $N$ tend to infinity, we obtain that $-\partial_h \phi(t,0) \le 0$, as desired.

\smallskip

\emph{Step 3.} We now consider the remaining possible contact points. Let $t, h > 0$ and $\phi \in C^\infty((0,\infty)\times [0,\infty))$ be such that $f-\phi$ has a local maximum at the point $(t,h)$. For the remainder of this proof, we allow the value of the constant $C < \infty$ to change from place to place, and to depend on $t$, $h$, $f$ and $\phi$, without further notice. For convenience, we introduce the notation
\begin{equation}  
\label{e.def.delta}
\de_N := \Ll\| \bar F_N - f \Rr\|_{L^\infty([0,t+1]\times [0,h+1])}^\frac 1 4 + N^{-\frac 1 {12}},
\end{equation}
and
\begin{equation*}  
\td \phi(t',h') := \phi(t,h) + (t-t')^2 + (h-h')^2.
\end{equation*}
We clearly have that $f-\td \phi$ has a strict local maximum at $(t,h)$. We also have that for every $(t',h') \in [0,t+1]\times [0,h+1]$,
\begin{equation*}  
\Ll( \bar F_N - \td \phi \Rr) (t',h') \le (f-\phi)(t',h') - (t-t')^2 - (h-h')^2 + \de_N^4,
\end{equation*}
while
\begin{equation*}  
\Ll( \bar F_N - \td \phi \Rr) (t,h) \ge (f-\phi)(t,h) - \de_N^4.
\end{equation*}

Since $f-\phi$ has a local maximum at $(t,h)$, we infer that for $N$ sufficiently large, the function $\bar F_N - \td \phi$ has a local maximum at $(t_N,h_N)$ satisfying
\begin{equation}
\label{e.bound.tNhN}
(t-t_N)^2 + (h-h_N)^2 \le 2 \de_N^4.
\end{equation}
The point of replacing $\phi$ by $\td \phi$ was precisely to obtain such an explicit estimate. We have
\begin{equation}  
\label{e.identity.FNphi}
\partial_h (\bar F_N - \td \phi)(t_N,h_N) = 0, \qquad \partial_t (\bar F_N - \td \phi)(t_N,h_N) = 0 .
\end{equation}
We next wish to use Proposition~\ref{p.approx.hj} to conclude. However, since the concentration result in Lemma~\ref{l.concentration} applies to $F_N - \bar F_N$ rather than its derivatives in $h$, we will want to take a small local average in the $h$ variable to control the term involving $\partial_h (F_N - \bar F_N)$. In preparation for this, we show in this step that there exists a constant $C < \infty$ such that for every $h' \in \R$ satisfying $|h'-h_N| \le C^{-1}$, we have
\begin{equation}  
\label{e.C11}
\Ll| \bar F_N(t_N,h') - \bar F_N(t_N,h_N) - (h'-h_N)\partial_h \bar F_N(t_N,h_N) \Rr| \le C (h'-h_N)^2.
\end{equation}
We start by writing Taylor's formula
\begin{multline}  
\label{e.Taylor}
\bar F_N(t_N,h') - \bar F_N(t_N,h_N) 
\\
= (h'-h_N) \partial_h \bar F_N(t_N,h_N) + \int_{h_N}^{h'} (h'-h'') \partial_h^2 \bar F_N(t_N,h'') \, \d h''.
\end{multline}
The same identity also holds with $\bar F_N$ replaced by $\td \phi$. Since $\bar F_N - \td \phi$ has a local maximum at $(t_N,h_N)$, and in view of \eqref{e.identity.FNphi}, we get that for $|h'-h_N|\le C^{-1}$,
\begin{equation*}  
\int_{h_N}^{h'} (h'-h'') \partial_h^2 \bar F_N(t_N,h'') \, \d h'' \le \int_{h_N}^{h'} (h'-h'') \partial_h^2 \td \phi(t_N,h'') \, \d h''.
\end{equation*}
Moreover, the integral on the right side is bounded by $C(h'-h_N)^2$, since $\td \phi$ is assumed to be smooth. By \eqref{e.conv.F}, we also have that $\partial_h^2 \bar F_N \ge -C$, and thus
\begin{equation*}  
\int_{h_N}^{h'} (h'-h'') \Ll|\partial_h^2 \bar F_N(t_N,h'')\Rr| \, \d h'' \le C (h'-h_N)^2.
\end{equation*}
Inequality~\eqref{e.C11} then follows using \eqref{e.Taylor} once more.

\smallskip

\emph{Step 4.}
We set
\begin{equation*}  
G_N(t',h') := \de_N^{-1} \int_{h'}^{h'+\delta_N} \bar F_N(t',h'') \, \d h''.
\end{equation*}
It is clear that the function $G_N$ converges to $f$ locally uniformly as $N$ tends to infinity. Hence, there exists a sequence $t_N',h_N' > 0$ such that for every $N$ sufficiently large, the function $G_N - \td \phi$ has a local maximum at $(t_N',h_N')$. Repeating the argument of the previous step, we also obtain that 
\begin{equation}  
\label{e.bound.tNhN'}
(t-t_N')^2 + (h-h_N')^2 \le 2 \de_N^4.
\end{equation}
We note that
\begin{equation}  
\label{e.identity.GNphi}
\partial_h (G_N - \td \phi)(t_N',h_N') = 0, \qquad \partial_t (G_N - \td \phi)(t_N',h_N') = 0 ,
\end{equation}
and
\begin{equation}  
\label{e.ineq.GNphi}
\partial_h^2 (G_N - \td \phi)(t_N',h_N') \le 0 .
\end{equation}
Recall from Proposition~\ref{p.approx.hj} that for every $h' > 0$,
\begin{multline}  
\label{e.remind.p.approx}
\Ll(\partial_t \bar F_N - 2 (\partial_h \bar F_N)^2 \Rr)(t_N',h') 
 \le 
\frac 2 {N}  \partial_{h}^2 \bar F_N(t_N',h') 
\\
+ 2 \E \Ll[ \Ll(\partial_h F_N - \partial_h \bar F_N\Rr)^2(t_N',h') \Rr]  +\frac{C}{N}\Ll(1 + \frac 1 {h'}\Rr).
\end{multline}
In the next two steps, we will show the following estimates:
\begin{equation}
\label{e.error.square}
\int_{h_N'}^{h_N'+ \de_N} \Ll( \partial_h \bar F_N(t_N',h') - \partial_h G_N(t_N',h_N')\Rr)^2  \, \d h' \le C \de_N^2,
\end{equation}
and
\begin{equation}
\label{e.average.concentration}
\int_{h_N'}^{h_N' + \de_N} \E \Ll[ \Ll(\partial_h F_N - \partial_h \bar F_N\Rr)^2(t_N',h') \Rr] \, \d h' \le C \de_N^2.
\end{equation}
For now, we assume that these estimates hold and show how to conclude. Using the fact that $\partial_h \bar F_N$ is bounded, Jensen's inequality, and \eqref{e.error.square}, we obtain that 
\begin{multline}  
\label{e.comp.square}
\Ll|(\partial_h G_N)^2(t_N',h_N') - \de_N^{-1}\int_{h_N'}^{h_N' + \de_N} (\partial_h \bar F_N)^2(t_N',h') \, \d h' \Rr| 
\\
\le C \de_N^{-1} \int_{h_N'}^{h_N'+ \de_N} \Ll| \partial_h \bar F_N(t_N',h') - \partial_h G_N(t_N',h_N')\Rr|  \, \d h'  \le C \sqrt{\de_N}.
\end{multline}
Averaging over $h' \in [h_N',h_N' + \de_N]$ in~\eqref{e.remind.p.approx}, using the estimate above and \eqref{e.average.concentration}, we get
\begin{equation*}  
\Ll(\partial_t G_N - 2 (\partial_h G_N)^2 \Rr)(t_N',h_N') 
 \le 
\frac 2 {N}  \partial_{h}^2 G_N(t_N',h_N') 
   + C \sqrt{\de_N}.
\end{equation*}
Appealing to \eqref{e.identity.GNphi}-\eqref{e.ineq.GNphi}, passing to the limit $N \to \infty$ and recalling that the first derivatives of $\phi$ and $\td \phi$ coincide at $(t,h)$, we conclude that 
\begin{equation}  
\label{e.concl.subsol}
\Ll( \partial_t \phi - 2(\partial_h \phi)^2 \Rr) (t,h) \le 0,
\end{equation}
as desired.

\smallskip

\emph{Step 5.} In order to complete the proof, there remains to show \eqref{e.error.square} and \eqref{e.average.concentration}. In this step, we prove \eqref{e.average.concentration}. The argument relies on the fact that, by integration by parts, we have for any smooth function $g \in C^\infty([a,b],\R)$ that
\begin{align}  
\notag
\|g'\|_{L^2(a,b)}^2
& = g(b)g'(b) - g(a) g'(a) - \int_a^b g g''
\\
\label{e.interpol}
& \le g(b)g'(b) - g(a) g'(a) + \|g\|_{L^\infty(a,b)} \, \|g''\|_{L^1(a,b)}.
\end{align}
Applying \eqref{e.interpol} with $g = F_N - \bar F_N$, using that $\partial_h F_N$ is bounded and the Cauchy-Schwarz inequality, we get that the left side of \eqref{e.average.concentration} is bounded by
\begin{multline*}  
\E \Ll[ \sup_{h' \in [h_N',h_N' + \de_N]} \Ll(F_N(t_N',h') - \bar F_N(t_N',h')\Rr)^2 \Rr] ^\frac 1 2 
\\
\times \Ll(C +  \E \Ll[ \Ll( \int_{h_N'}^{h_N' + \de_N} \Ll|\partial_h^2 (F_N - \bar F_N)(t_N',h') \Rr|\, \d h' \Rr)^2  \Rr] ^\frac 1 2\Rr).
\end{multline*}
By Lemma~\ref{l.concentration}, the first term in this product is bounded by $C N^{-\frac 1 6}$. For the second term, we use \eqref{e.conv.F} to observe that, for the constant $C = C_0$ identified there,
\begin{align*}  
\Ll|\partial_h^2 F_N(t_N',h') \Rr| 
& \le \Ll|\partial_h^2 F_N(t_N',h') + \frac{C_0|z|}{N^\frac 1 2 h^\frac 3 2}\Rr| + \frac{C_0|z|}{N^\frac 1 2 h^\frac 3 2}
\\
& = \partial_h^2 F_N(t_N',h') + \frac{2C_0|z|}{N^\frac 1 2 h^\frac 3 2} ,
\end{align*}
and thus, using again that $\partial_h F_N$ is bounded, we conclude that
\begin{equation*}  
\E \Ll[ \Ll( \int_{h_N'}^{h_N' + \de_N} \Ll|\partial_h^2 F_N(t_N',h') \Rr|\, \d h' \Rr)^2  \Rr]^\frac 1 2  \le C + \frac{C\de_N}{ h^\frac 3 2} \le C. 
\end{equation*}
Since $N^{-\frac 1 6} \le \de_N^2$, this completes the proof of \eqref{e.average.concentration}.

\smallskip

\emph{Step 6.} We now prove \eqref{e.error.square}.
Observe that
\begin{equation*}  
\partial_h G_N(t_N',h_N') = \de_N^{-1} \Ll(\bar F_N(t_N',h_N'+ \de_N) - \bar F_N(t_N',h_N')\Rr).
\end{equation*}
We use \eqref{e.interpol} with $g$ replaced by the function
\begin{equation*}  
g_N : h' \mapsto \bar F_N(t_N',h') - \bar F_N(t_N',h_N') + (h'-h_N')\partial_h G_N(t_N',h_N')
\end{equation*}
to get that
\begin{multline*}  
\int_{h_N'}^{h_N'+ \de_N} \Ll( \partial_h \bar F_N(t_N',h') - \partial_h G_N(t_N',h_N')\Rr)^2  \, \d h'
\\
 \le \|g_N\|_{L^\infty(h_N',h_N' + \de_N)} \,
\int_{h_N'}^{h_N' + \de_N} |\partial_h^2 \bar F_N(t_N',h')| \, \d h'.
\end{multline*}
Using also \eqref{e.conv.F} and \eqref{e.gradest.barF}, we obtain that
\begin{equation*}  
\int_{h_N'}^{h_N'+ \de_N} \Ll( \partial_h \bar F_N(t_N',h') - \partial_h G_N(t_N',h_N')\Rr)^2  \, \d h' \le C \|g_N\|_{L^\infty(h_N',h_N' + \de_N)}.
\end{equation*}
Since $\bar F_N$ is Lipschitz continuous in the variable $t$, we also have that 
\begin{multline*}  
\|g_N\|_{L^\infty(h_N',h_N' + \de_N)} \le C |t_N - t_N'| 
\\
+  \sup_{h' \in [h_N',h_N' + \de_N]} \Ll| \bar F_N(t_N,h') - \bar F_N(t_N,h_N') + (h'-h_N')\partial_h G_N(t_N,h_N') \Rr| .
\end{multline*}
The estimate \eqref{e.error.square} then follows using \eqref{e.bound.tNhN}, \eqref{e.bound.tNhN'} and \eqref{e.C11}.
\end{proof}

We conclude this section with some remarks on variational representations for the function $f$ appearing in Theorem~\ref{t.hj}. Solutions to Hamilton-Jacobi equations of the form $\partial_t f - \mathsf H(\nabla f)$ with convex (resp.\ concave) $\mathsf H$ have a variational representation given by the Hopf-Lax formula, in which the convex (resp.\ concave) dual of $\mathsf H$ appears (see e.g.~\cite[Theorem~10.3.4.3]{evans}). It is usually under this variational presentation that the limit free energy of mean-field statistical mechanics models is identified. In our case, the function~$\mathsf H$ is simply $p \mapsto 2p^2$, whose convex dual is $q \mapsto \frac {q^2}{8}$.
\begin{proposition}[Hopf-Lax formula]
\label{p.hopf-lax}
For every $t \ge 0$ and $h \ge 0$, we set
\begin{equation}  
\label{e.hopf-lax}
f(t,h) := \sup_{h' \ge 0} \Ll( \psi({h}') - \frac{({h} - {h}')^2}{8t} \Rr),
\end{equation}
with the understanding that $f(0,h) = \psi(h)$. The function $f$ is the unique viscosity solution of \eqref{e.hj} that satisfies $f(0,h) = \psi(h)$ and is globally Lipschitz continuous in the variable $h$.
\end{proposition}
For completeness, we provide a proof of this classical result in Appendix~\ref{s.visc}. Denoting $\mathsf H(p) := 2p^2$ and $\mathsf H^*(q) := \frac{q^2}{8}$, we have the following equivalent expressions for $f$ which may be of interest:
\begin{align*}  
f(t,h) & = \sup_{h' \ge 0} \Ll( \psi({h}') - t \mathsf H^* \Ll( \frac{{h} - {h}'}{t} \Rr)  \Rr) \\
& = \sup_{h' \ge 0} \, \inf_{p \in \R} \Ll( \psi(h') - t \Ll( p \frac{h-h'}{t} - \mathsf H(p) \Rr)  \Rr) \\
& = \sup_{h' \ge 0} \, \inf_{p \in \R} \Ll( \psi(h') - p (h-h') + t \mathsf H(p)  \Rr).
\end{align*}
We stress that the proof of Theorem~\ref{t.hj} does not require that $f$ be identified by such a variational presentation. Moreover, the analysis of $f$ itself does not necessarily require explicit usage of this formula. For instance, if one wants to observe that $\partial_h f(t,0) = 0$ for small values of $t \ge 0$, which at least on a heuristic level corresponds to a regime where there is no correlation between $x$ and $\bar x$, see \eqref{e.formula.partialmu2}, then we may proceed as follows. First, we check that there exists a constant $C < \infty$ such that for every $h \ge 0$, we have $\psi(h) \le C h^2$. (See \eqref{e.psi.first.derivative} for a first step.) We next observe that the function
\begin{equation*}  
(t,h) \mapsto \frac{C h^2}{1-8Ct}
\end{equation*}
is a supersolution of \eqref{e.hj} on $(0,(8C)^{-1})\times [0,\infty)$, and thus, by the comparison principle, the solution $f$ to \eqref{e.hj} remains below this supersolution. Since the null function is a subsolution, we deduce that $\partial_h f(t,0) = 0$ for every $t <(8C)^{-1}$.

%
%
%
%
%
%

\section{Extension to tensors}
\label{s.tensor}

We now explain how to adapt the method to tensors of arbitrary order. In this setting, the result was obtained in \cite{les,bm1}. One motivation for exploring this generalization is that some methods, such as that used in \cite{eak}, do not seem to generalize well to tensors of odd order. 

\smallskip

We fix an integer $p \ge 1$. Generalizing the previous setting, we consider the problem of estimating the vector $\bar x = (\bar x_1,\ldots, \bar x_N) \in \R^N$ given the observation of 
\begin{equation*}  
\frac{\sqrt{t}}{N^{\frac{p-1}{2}}} \, \bar x^{\otimes p}+ W, 
\end{equation*}
where $W = (W_{i_1 \ldots i_p})_{1 \le i_1,\ldots,i_p \le N}$ is now a tensor of order $p$ made of independent standard Gaussian random variables, independent of the vector $\bar x$, and where for any $x \in \R^N$, we denote by~$x^{\otimes p}$ the tensor of order $p$ such that, for every $i_1,\ldots,i_p \in \{1,\ldots,N\}$,
\begin{equation*}  
(x^{\otimes p})_{i_1 \ldots i_p} = x_{i_1} \, \cdots \ x_{i_p}.
\end{equation*}
We redefine $H_N(t,h,x)$ to be
\begin{multline*}  
H_N(t,{h},x) := 
\frac{\sqrt{t}}{N^{\frac{p-1}{2}}} W : x^{\otimes p} + \frac t {N^{p-1}} (x \cdot \bar x)^p - \frac t {2N^{p-1}} |x|^{2p}
\\
+ \sqrt{h} \, z \cdot x + {h} \, x \cdot \bar x - \frac {h} 2 |x|^2,
\end{multline*}
and set $F_N$ and $\bar F_N$ to be as in \eqref{e.enriched.free.energy} and \eqref{e.expectation.free.energy}. The analogue of Theorem~\ref{t.hj} in the context of tensors reads as follows.

\begin{theorem}[Convergence to HJ]
\label{t.hj.tensor}
For every $M \ge 1$, we have
\begin{equation*}  
\lim_{N\to \infty} \E \Ll[ \sup_{[0,M]^2} (F_N - f)^2 \Rr] = 0,
\end{equation*}
where $f(t,h)$ is the viscosity solution of the Hamilton-Jacobi equation 
\begin{equation}  
\label{e.hj.tensor}
\Ll\{
\begin{aligned}
 \partial_t f - 2^{p-1}(\partial_{h} f)^p & = 0 \quad && \text{in } (0,+\infty)^2, \\
 - \partial_h f & = 0 \quad && \text{on }  (0,+\infty) \times \{0\},
\end{aligned}
\Rr.
\end{equation}
with initial condition $f(0,h) = \psi(h)$. 
\end{theorem}
The next proposition is our replacement for Proposition~\ref{p.approx.hj}. \begin{proposition}[Approximate HJ in finite volume]
There exists $C < \infty$ such that for every $N \ge 1$ and uniformly over $[0,\infty)^2$,
\label{p.approx.hj.tensor}
\begin{equation*}  
\Ll|\partial_t \bar F_N - 2^{p-1}(\partial_{h} \bar F_N)^p\Rr|^2 \le \frac{C}{N}  \partial_{h}^2 \bar F_N  + C\E \Ll[ \Ll(\partial_h F_N - \partial_h \bar F_N\Rr)^2 \Rr]  +\frac{C}{N}\Ll(\frac 1 h + \frac 1 {\sqrt{h}}\Rr),
\end{equation*}
and moreover, 
\begin{equation}  
\label{e.boundary.tensor}
\partial_h \bar F_N\ge 0.
\end{equation}
\end{proposition}
\begin{proof}[Proof of Proposition~\ref{p.approx.hj.tensor}]
Observe that
\begin{equation}  
\label{e.formula.partialt.tensor}
\partial_t  F_N(t,{h}) = \frac 1 N\la \frac {1} {2N^{\frac{p-1}{2}}\sqrt {t}} \,  W : x^{\otimes p} +  \frac 1 {N^{p-1}} (x\cdot \bar x)^p - \frac 1 {2N^{p-1}} |x|^{2p} \ra.
\end{equation}
By Gaussian integration by parts and the Nishimori identity, we deduce that
\begin{equation*}  
\partial_t \bar F_N(t,{h}) = \frac 1 {2N^{p}} \E \la (x \cdot \bar x)^p \ra.
\end{equation*}
The expressions \eqref{e.formula.partialmu1}-\eqref{e.formula.partialmu2} are still valid (as well as \eqref{e.positivity}). We deduce that
\begin{equation}  
\label{e.variance.identity.tensor}
2\partial_t \bar F_N - 2^{p} (\partial_{h} \bar F_N)^p  = \E \la \Ll(\frac{x \cdot \bar x}{N}\Rr)^p\ra   - \Ll( \E \la \frac{x \cdot \bar x}{N}\ra  \Rr)^p .
\end{equation}
Using that $a^p - b^p = (a-b)(a^{p-1} + \cdots + b^{p-1})$ and the fact that the support of the measure $P$ is bounded, we get that
\begin{equation*}  
\Ll|\partial_t \bar F_N - 2^{p-1} (\partial_{h} \bar F_N)^p \Rr|^2 \le \frac{C}{N^2} \E \la \Ll( x\cdot \bar x - \E\la x\cdot \bar x \ra \Rr)^2 \ra.
\end{equation*}
The arguments in the proof of Proposition~\ref{p.approx.hj} apply without any modification to show that
\begin{align*}  
\E \la \Ll( x\cdot \bar x - \E\la x\cdot \bar x \ra \Rr)^2 \ra  & \le  4 \E \la \big(  H_N'(h,x) - \E \la H_N'(h,x) \ra \big) ^2 \ra + \frac{CN}{\sqrt{h}} \\
& \le 4N \partial_h^2 \bar F_N(t,h) + 4N^2 \, \E \Ll[ \Ll( \partial_h F_N(t,h) - \partial_h \bar F_N(t,h) \Rr) ^2 \Rr] 
\\
&  \qquad+ C N (h^{-1} + h^{-\frac 1 2}).
\end{align*}
Combining the two previous displays yields Proposition~\ref{p.approx.hj.tensor}.
\end{proof}
\begin{proof}[Proof of Theorem~\ref{t.hj.tensor}]
As in the proof of Theorem~\ref{t.hj}, it suffices to show that if $f$ is such that $\bar F_N$ converges locally uniformly to $f$ along a subsequence, then $f$ is a viscosity solution of \eqref{e.hj.tensor}. Abusing notation, we do not write explicitly the subsequence along which the convergence holds.

\smallskip

\emph{Step 1.} We show that $f$ is a viscosity subsolution of \eqref{e.hj.tensor}. The proof follows Steps 2-6 of the proof of Theorem~\ref{t.hj} very closely. The first difference is that we use Proposition~\ref{p.approx.hj.tensor} to replace \eqref{e.remind.p.approx} by
\begin{multline}
\label{e.remind.p.approx.tensor}
\Ll(\partial_t \bar F_N - 2^{p-1} (\partial_h \bar F_N)^p \Rr)(t_N',h') 
 \le 
\Bigg(\frac C {N}  \partial_{h}^2 \bar F_N(t_N',h') 
\\
+ C \E \Ll[ \Ll(\partial_h F_N - \partial_h \bar F_N\Rr)^2(t_N',h') \Rr]  +\frac{C}{N}\Ll(1 + \frac 1 {h'}\Rr)\Bigg)^\frac 1 2.
\end{multline}
(Implicit in this expression is the fact that the quantity under the square root on the right side is nonnegative.) The estimates \eqref{e.error.square} and \eqref{e.average.concentration} still hold and the proofs given there apply without any modification. We deduce as in \eqref{e.comp.square} that
\begin{equation*}  
\Ll|(\partial_h G_N)^p(t_N',h_N') - \de_N^{-1}\int_{h_N'}^{h_N' + \de_N} (\partial_h \bar F_N)^p(t_N',h') \, \d h' \Rr| 
\le C \sqrt{\de_N}.
\end{equation*}
We average the inequality \eqref{e.remind.p.approx.tensor} over $h' \in [h_N',h_N' + \de_N]$, use Jensen's inequality, the estimate above and \eqref{e.average.concentration} to obtain that
\begin{equation*}  
\Ll(\partial_t G_N - 2^{p-1} (\partial_h G_N)^p \Rr)(t_N',h_N') 
 \le 
\Ll(\frac C {N}  \partial_{h}^2 G_N(t_N',h_N') 
+ C \de_N \Rr)^\frac 1 2 + C \sqrt{\de_N},
\end{equation*}
and then conclude as before that \eqref{e.concl.subsol} holds.

\smallskip

\emph{Step 2.} We now show that $f$ is a viscosity supersolution of \eqref{e.hj.tensor}. Let $(t,h) \in (0,\infty)\times [0,\infty)$ and $\phi \in C^\infty((0,\infty)\times [0,\infty))$ be such that $f-\phi$ has a strict local minimum at the point $(t,h)$. We keep the definition of $\de_N$ as in~\eqref{e.def.delta} for consistency of notation (although here a simpler choice not depending on the rate of convergence of $\bar F_N$ to $f$ would also do), and redefine $G_N$ to be
\begin{equation*}  
G_N(t',h') = \de_N^{-1} \int_{h'+\de_N}^{h'+2\de_N} \bar F_N(t',h'') \, \d h''.
\end{equation*}
In this new definition of $G_N$, we have shifted the interval over which the integral is taken by $\de_N$ to the right in order to avoid the singularity of the error term in Proposition~\ref{p.approx.hj.tensor} near $h = 0$.
Since $G_N$ converges to $f$ locally uniformly, there exists a sequence $(t_N,h_N) \in (0,\infty)\times [0,\infty)$ converging to $(t,h)$ as $N$ tends to infinity such that $G_N - \phi$ has a local minimum at $(t_N,h_N)$. By Proposition~\ref{p.approx.hj.tensor}, for every $h' > 0$,
\begin{multline}  
\label{e.remind.again}
\Ll(\partial_t \bar F_N - 2^{p-1} (\partial_h \bar F_N)^p \Rr)(t_N,h') 
 \ge 
-\Bigg(\frac C {N}  \partial_{h}^2 \bar F_N(t_N,h') 
\\
+ C \E \Ll[ \Ll(\partial_h F_N - \partial_h \bar F_N\Rr)^2(t_N,h') \Rr]  +\frac{C}{N}\Ll(1 + \frac 1 {h'}\Rr)\Bigg)^\frac 1 2.
\end{multline}
We also observe that the estimate \eqref{e.average.concentration} still holds in the present context. Averaging the inequality \eqref{e.remind.again} over $h' \in [h_N + \de_N,h_N + 2\de_N]$, using Jensen's inequality, and \eqref{e.average.concentration}, we get that
\begin{equation*}  
\Ll(\partial_t G_N - 2^{p-1} (\partial_h G_N)^p \Rr)(t_N,h_N) 
 \ge 
-\Ll(\frac C {N}  \partial_{h}^2 G_N(t_N,h_N) 
+ C \de_N  +\frac{C}{N\de_N}\Rr)^\frac 1 2.
\end{equation*}
Using Jensen's inequality for the left side of \eqref{e.remind.again} is justified since $\partial_h \bar F_N \ge 0$. Since $\partial_h \bar F_N$ is bounded, we have that $\Ll|\partial_h^2 G_N(t_N,h_N)\Rr| \le \de_N^{-1}$. We thus obtain that
\begin{equation}  
\label{e.lim.nonneg}
\liminf_{N\to \infty} \Ll(\partial_t G_N - 2^{p-1} (\partial_h G_N)^p \Rr)(t_N,h_N) \ge 0.
\end{equation}
If $h_N > 0$ for infinitely many values of $N$, then the first derivatives of $G_N$ and $\phi$ coincide at $(t_N,h_N)$ for these values of $N$, and we thus deduce from \eqref{e.lim.nonneg} that
\begin{equation*}  
\Ll(\partial_t \phi - 2^{p-1} (\partial_h \phi)^p \Rr)(t,h) = \lim_{N\to \infty} \Ll(\partial_t \phi - 2^{p-1} (\partial_h \phi)^p \Rr)(t_N,h_N)\ge 0.
\end{equation*}
On the other hand, if $h_N = 0$ for infinitely many values of $N$, then we can reproduce the argument of Step 1 of the proof of Theorem~\ref{t.hj.tensor} to conclude. Indeed, in this case, we must have $h = 0$,  
\begin{equation}  
\label{e.compare.grad.bdry.tensor}
\partial_h (G_N - \phi)(t_N,h_N) \ge 0, \quad \text{and} \quad \partial_t (G_N - \phi)(t_N,h_N) = 0 .
\end{equation}
If $-\partial_h \phi(t,h) \ge 0$, then there is nothing to show. Else, using the first statement in \eqref{e.compare.grad.bdry.tensor}, we have that
\begin{equation*}  
(\partial_h \phi)^p(t,h) = \lim_{N \to \infty} (\partial_h \phi)^p(t_N,h_N) \le \liminf_{N \to \infty} (\partial_h \bar F_N)^p(t_N,h_N),
\end{equation*}
and thus, using the second statement in \eqref{e.compare.grad.bdry.tensor} and then \eqref{e.lim.nonneg}, we deduce that
\begin{equation*}  
\Ll(\partial_t \phi - 2^{p-1} (\partial_h \phi)^p \Rr)(t,h) \ge \limsup_{N\to \infty} \Ll(\partial_t \bar F_N - 2^{p-1} (\partial_h \bar F_N)^p \Rr)(t_N,h_N) \ge 0,
\end{equation*}
thereby completing the proof.
\end{proof}

\appendix 

%
%
%
%
%
%

\section{Nishimori identity}
\label{s.nishi}

We verify the Nishimori identity stated in \eqref{e.nishi} and \eqref{e.nishi2}. We redefine the variable $Y$ to be
\begin{equation*}  
Y = (Y^{(1)},Y^{(2)}) := \Ll( \sqrt{\frac{t}{N}} \, \bar x \, \bar x^\t + W, \  \sqrt{h} \, \bar x + z\Rr).
\end{equation*}
For every bounded measurable functions $f$ and $g$, we can write the quantity $\E \Ll[ f(\bar x) g(Y) \Rr]$, up to a normalization constant that depends neither on $f$ nor on $g$, as
\begin{equation*}
 \int f(x) g \Ll(  \sqrt{\frac{t}{N}} \, x \, x^\t + W, \  \sqrt{h} \, x + z \Rr) \, \exp \Ll( - \sum_{i,j = 1}^N \frac{W_{ij}^2}2 - \sum_{i = 1}^N \frac{z_i^2}{2} \Rr) \, \d W \, \d z \, \d P_N(x),
\end{equation*}
with the shorthand notation $dW := \prod_{i,j} dW_{ij}$ and $dz := \prod_i dz_i$. A change of variables leads to 
\begin{multline*}  
\int f(x) g(Y^{(1)},Y^{(2)}) \\
 \exp \Ll( -\frac 1 2 \sum_{i,j = 1}^N \Ll( Y^{(1)}_{ij} - \sqrt{\frac t N} x_i x_j \Rr)^2 - \sum_{i= 1}^N \frac{(Y^{(2)}_i - \sqrt{h} \, x)^2}{2} \Rr) \, \d Y \, \d P_N(x). 
\end{multline*}
Denoting the exponential factor above by $\mcl E(x,Y)$, we thus obtain that the law of $Y$ is the law with density given, up to a normalization constant, by
\begin{equation*}  
\bar {\mcl E}(Y) := \int \mcl E(x,Y) \, \d P_N(x),
\end{equation*}
and that
\begin{equation}  
\label{e.joint.law}
\E \Ll[ f(\bar x) g(Y) \Rr] = \int f(x) \frac{\mcl E(x,Y)}{\bar {\mcl E}(Y)} \, \d P_N(x)  \, g(Y) \, \bar {\mcl E}(Y) \, \d Y.
\end{equation}
The conditional law of $\bar x$ given $Y$ is thus the probability measure given by
\begin{equation*}  
\frac{\mcl E(x,Y)}{\bar{\mcl E}(Y)} \, \d P_N(x).
\end{equation*}
Moreover, a calculation similar to that in \eqref{e.def.H1} yields that this quantity can be rewritten as
\begin{equation*}  
\frac{e^{H_N(t,h,x)} \, \d P_N(x)}{\int_{\R^N} e^{H_N(t,h,x')} \, \d P_N(x')},
\end{equation*}
which is the Gibbs measure defined in \eqref{e.def.Gibbs}. We denote by $x^{(1)}, \ldots, x^{(k)}$ a sequence of $k$ random variables which, conditionally on $\bar x$, $W$ and $z$, are independent and distributed according to this measure; we denote their joint (conditional) law by $\la \cdot \ra$. We thus have that, for every $k \in \N$ and bounded measurable function $f (x^{(1)},\ldots,x^{(k)},Y)$,
\begin{multline*}  
\E \la f(x^{(1)},\ldots,x^{(k)},Y) \ra = 
\int f(x^{(1)},\ldots,x^{(k)},Y) \\
\times \frac{\mcl E(x^{(1)},Y)}{\bar{\mcl E}(Y)} \, \d P_N(x^{(1)}) \ \cdots \ \frac{\mcl E(x^{(k)},Y)}{\bar{\mcl E}(Y)} \, \d P_N(x^{(k)}) \, \bar{\mcl E}(Y) \, \d Y.
\end{multline*}
In view of the expression for the joint law of $\bar x$ and $Y$ obtained in \eqref{e.joint.law}, we deduce that
\begin{equation*}  
\E \la f(x^{(1)},\ldots,x^{(k)},Y) \ra = \E \la f(x^{(1)},\ldots,x^{(k-1)},\bar x,Y) \ra.
\end{equation*}
This implies in particular that \eqref{e.nishi} and \eqref{e.nishi2} hold.

%
%
%
%
%
%

\section{Classical results on viscosity solutions}
\label{s.visc}

In this appendix, for the reader's convenience, we prove the comparison principle (Proposition~\ref{p.comparison}) and the Hopf-Lax formula (Proposition~\ref{p.hopf-lax}) for solutions of the Hamilton-Jacobi equation \eqref{e.hj}. Classical references for such results include \cite{evans,guide}. 
\begin{proof}[Proof of Proposition~\ref{p.comparison}]
We argue by contradiction, assuming instead that
\begin{equation}  
\label{e.max.int}
\sup_{[0,\infty)^2} (u-v) > \sup_{\{0\}\times [0,\infty)} (u-v).
\end{equation}
The argument rests on the idea of ``doubling the variables'' and considering the maximization of functions of the form
\begin{equation}  
\label{e.function.present}
\Ll((t,h),(t',h')\Rr) \mapsto u(t,h) - v(t',h') - \frac 1 {2\al} \Ll( |t-t'|^2 + |h-h'|^2 \Rr) ,
\end{equation}
where $\alpha > 0$ is a parameter that is ultimately sent to $0$.
We will decompose this argument into four steps. In Step 1, we modify the functions $u$ and $v$ slightly so that they become strict sub- and supersolutions respectively. In Step 2, we modify the function $u$  further to ensure that the maximum of the function in \eqref{e.function.present} is achieved at a point that remains in a bounded set as $\al \to 0$. In a preliminary Step 0, we build a convenient special function for this purpose. The conclusion is then derived in Step 3.

\smallskip

\emph{Step 0.} We build a special function $\Phi_\de \in C^\infty([0,\infty)^2)$ such that for every $T,H > 0$ and $\delta > 0$ sufficiently small, the following properties hold.
\begin{equation}  
\label{e.psi.local.negligible}
\|\Phi_\delta\|_{L^\infty([0,T]\times [0,H])}  + \|\nabla \Phi_\delta\|_{L^\infty([0,T]\times [0,H])} \le \de,
\end{equation}
\begin{equation}  
\label{e.psi.lower}
\forall t\ge 0, \ \forall  h \ge 2 \delta^{-2} , \qquad \Phi_\delta(t,h) \ge \delta^{-1} h,
\end{equation}
and
\begin{equation}  
\label{e.bound.der}
\partial_t \Phi_\delta \ge \delta^{-\frac 1 2} |\partial_h \Phi_\delta| \qquad \text{on } [0,T] \times [0,\infty).
\end{equation}
Let $\chi \in C^\infty_c(\R)$ be a smooth function satisfying $0 \le \chi \le 1$ and such that $\chi \equiv 0$ on $(-\infty,0]$ and $\chi \equiv 1$ on $[1,\infty)$. For every $z \in \R$, we set
\begin{equation*}  
\sigma(z) := \int_{-\infty}^z e^{z'} \, (1-\chi(z')) \, \d z'.
\end{equation*}
Note that for every $z \in (-\infty,0]$, we have that $\sigma(z) = e^{z}$ and that $\sigma$ is constant (and $\sigma \ge 1$) on $[1,\infty)$. Moreover, there exists a  constant $C < \infty$ such that 
\begin{equation}  
\label{e.si'si}
0 \le \sigma' \le C \sigma 
.
\end{equation}
For every $\delta \in (0,1]$, we consider the function
\begin{equation*}  
\Phi_\delta(t,h) :=  \delta^{-1} h  \sigma(\delta h - \de^{-1}) +\delta^{-3} t \sigma(\de h - \de^{-1}). 
\end{equation*}
Roughly speaking, the function $h \mapsto \si(\de h - \de^{-1})$ serves as a smoothed indicator function for the set $\{h \ge \de^{-2}\}$.
The properties \eqref{e.psi.local.negligible} and \eqref{e.psi.lower} are immediate. For the last property, we observe that
\begin{equation*}  
\partial_t \Phi_\delta(t,h) = \de^{-3} \sigma(\de h - \de^{-1}),
\end{equation*}
\begin{equation*}  
\partial_h \Phi_\delta(t,h) = \delta^{-1} \sigma(\de h - \de^{-1}) + h \si'(\de h - \de^{-1}) + \de^{-2} t \si'(\de h - \de^{-1}) .
\end{equation*}
Since $\sigma'$ is supported in $(-\infty,1]$, the second term on the right side vanishes whenever $h \le \de^{-2} + \de^{-1} \le 2 \de^{-2}$. We deduce that 
\begin{equation*}  
0 \le \partial_h \Phi_\delta(t,h) \le  \delta^{-1} \sigma(\de h - \de^{-1}) + 2\de^{-2} \si'(\de h -\de^{-1}) + \de^{-2} t \si'(\de h - \de^{-1}).
\end{equation*}
The inequality \eqref{e.bound.der} thus follows using \eqref{e.si'si}.

\smallskip

\emph{Step 1.} We show that without loss of generality, we can assume that there exists $\ep > 0$ such that $u$ is a viscosity solution of
\begin{equation}
\label{e.hj.sub.ep}
\Ll\{
\begin{aligned}
 \partial_t u - 2(\partial_{h} u)^2 & \le -\ep \quad && \text{in } (0,+\infty)^2, \\
 - \partial_h u & \le -\ep \quad && \text{on }  (0,+\infty) \times \{0\}.
\end{aligned}
\Rr.
\end{equation}
Indeed, since we assume that $\partial_h u \in L^\infty([0,\infty)^2)$, it suffices to replace $u$ by $u_\ep := u + \eps h - C \eps t$ for a sufficiently large constant $C$ depending on $\|\partial_h u\|_{L^\infty([0,\infty)^2)}$, and then select $\ep > 0$ sufficiently small that the property \eqref{e.max.int} still holds. Similarly, we can assume that the function $v$ is a viscosity solution of
\begin{equation}
\label{e.hj.sup.ep}
\Ll\{
\begin{aligned}
 \partial_t v - 2(\partial_{h} v)^2 & \ge \ep \quad && \text{in } (0,+\infty)^2, \\
 - \partial_h v & \ge \ep \quad && \text{on }  (0,+\infty) \times \{0\}.
\end{aligned}
\Rr.
\end{equation}
Note that these modifications preserve the fact that $u$ and $v$ are uniformly Lipschitz continuous in $h$. 

\smallskip

\emph{Step 2.} 
We now ``localize'' the function $u$, in the sense that we make sure that the function becomes very negative as me move away from a bounded set.  To start with, we can replace $u$ by $u-\frac{\ep}{T-t}$ for some $\ep > 0$ and $T \in (0,\infty)$. This preserves the fact that $u$ solves \eqref{e.hj.sub.ep} on $(0,T)\times [0,\infty)$, and for $\ep > 0$ sufficiently small and $T$ sufficiently large, it also preserves the property \eqref{e.max.int} in the sense that
\begin{equation}  
\label{e.max.int2}
\sup_{[0,T) \times [0,\infty)^2} (u-v) > \sup_{\{0\}\times [0,\infty)} (u-v).
\end{equation}
This modification of $u$ also ensures that for every $H > 0$, 
\begin{equation}  
\label{e.u.diverge.time}
\lim_{\eta \to 0} \sup \{u(t,h) \ : (t,h) \in [ T - \eta,T) \times [0 , H]\} = -\infty.
\end{equation}
Note that we still have that 
\begin{equation*}  
\sup_{t \in [0,T)} u(t,0) < + \infty,
\end{equation*}
and that $u$ is uniformly Lipschitz continuous in $h$, and thus there exists a constant $C < \infty$ such that
\begin{equation}  
\label{e.u.still.upper}
\forall (t,h) \in [0,T)\times [0,\infty), \qquad u(t,h) \le C(1+h).
\end{equation}
Our final modification is to replace $u$ by $u_\de := u-\Phi_\de$ for the function $\Phi_\de$ defined in the Step 0. It is clear from \eqref{e.psi.local.negligible} that for $\delta > 0$ sufficiently small, the properties \eqref{e.max.int2} and \eqref{e.u.diverge.time} still hold with $u$ replaced by $u_\de$. It is also clear from \eqref{e.psi.lower} and \eqref{e.u.still.upper} that
for every $\de > 0$ sufficiently small, we have
\begin{equation}  
\label{e.u.goes.down}
\forall (t,h) \in [0,T)\times [0,+\infty),  \qquad \ u_\de(t,h) \le C(1+h)- \de^{-1} h \1_{\{h \ge 2\de^{-2}\}}.
\end{equation}
There remains to verify that $u_\delta$ is still a solution of \eqref{e.hj.sub.ep}, possibly after replacing $\ep$ by $\ep/2$. Formally, the verification of the boundary condition on $[0,T)\times \{0\}$ is immediate from \eqref{e.psi.local.negligible}, while we have
\begin{equation*}  
\partial_t u_\delta - 2 (\partial_h u_\delta)^2 = \partial_t u - 2 (\partial_h u)^2 - \partial_t \Phi_\delta - 4 \partial_h u \, \partial_h \Phi_\delta - 2 (\partial_h \Phi_\delta)^2,
\end{equation*}
and since $\partial_h u$ is bounded, it follows from \eqref{e.bound.der} that for every $\delta \ge 0$ sufficiently small,
\begin{equation*}  
\partial_t \Phi_\delta + 4 \partial_h u \, \partial_h \Phi_\delta + 2 (\partial_h \Phi_\delta)^2 \ge 0.
\end{equation*}
This formal calculation is easily made rigorous using test functions. 

\smallskip

Finally, we observe that as a consequence of \eqref{e.u.goes.down} and the Lispschitz continuity of $v$ in the $h$ variable, there exists a constant $C < \infty$ such that for every $t\in [0,T)$, $h,t',h'\ge 0$, we have
\begin{equation*}  
u_\de(t,h) - v(t',h') \le C(1+h+h') - \de^{-1} h \1_{\{h \ge 2\de^{-2}\}}.
\end{equation*}
We select $\de > 0$ sufficiently small that for $t,h,t',h'$ as above,
\begin{equation}
\label{e.bound.uv}
u_\de(t,h) - v(t',h') \le C\Ll(1+h+h'-2  h \1_{\{h \ge 2\de^{-2}\}}\Rr).
\end{equation}

\smallskip

\emph{Step 3.}
Summarizing the result of the previous steps, we have shown that without loss of generality, we can assume that that $v$ solves \eqref{e.hj.sub.ep} on $(0,\infty)\times [0,\infty)$, that $u$ solves \eqref{e.hj.sup.ep} on $(0,T)\times [0,\infty)$ and satisfies \eqref{e.u.diverge.time} and \eqref{e.bound.uv} for every $t \in [0,T)$ and $h,t',h' \ge 0$, and that \eqref{e.max.int2} holds. 

\smallskip

We define, for every $\alpha \in (0,1]$ and $t \in [0,T)$, $h,t',h' \ge 0$,
\begin{equation*}  
\Psi_\alpha(t,h,t',h') := u(t,h) - v(t',h') - \frac 1 {2\al} \Ll( |t-t'|^2 + |h-h'|^2 \Rr) .
\end{equation*}
We claim that the maximum of $\Psi_\alpha$ is achieved at some point $(t_\al,h_\al,t'_\al,h'_\al)$, and that this point remains in a bounded set as $\al > 0$ is sent to $0$. For fixed $\al > 0$, consider a sequence of approximate maximizers for $\Psi_\alpha$ denoted by $(t_{n,\al},h_{n,\al},t'_{n,\al},h'_{n,\al})$. We deduce from \eqref{e.bound.uv} that for some $C < \infty$,
\begin{equation*}  
C\Ll(1+h_{n,\al}+h'_{n,\al}-2  h_{n,\al} \1_{\{h_{n,\al} \ge 2\de^{-2}\}}\Rr) 
- \frac 1 {2\al} \Ll(  |t_{n,\al}-t'_{n,\al}|^2 + |h_{n,\al}-h'_{n,\al}|^2 \Rr)  \ge -C,
\end{equation*}
and in particular, 
\begin{equation*}  
\frac 1 2 \Ll(  |t_{n,\al}-t'_{n,\al}|^2 + |h_{n,\al}-h'_{n,\al}|^2 \Rr) 
- C\Ll(1+h_{n,\al}+h'_{n,\al}-2  h_{n,\al} \1_{\{h_{n,\al} \ge 2\de^{-2}\}}\Rr)  \le C.
\end{equation*}
If $h_{n,\al} < 2 \de^{-2}$, then we can obtain a uniform upper bound on $h'_{n,\al}$, and thus also on $|t_{n,\al}-t'_{n,\al}|$. Otherwise, we can first obtain a uniform upper bound on $|h_{n,\al}-h'_{n,\al}|$, and then deduce an upper bound on $h_{n,\al}$, $h'_{n,\al}$ and $|t_{n,\al}-t'_{n,\al}|$. Finally, we can use \eqref{e.u.diverge.time} to conclude that the maximizer of~$\Psi_\alpha$ exists and remains in a bounded set as $\al$ tends to $0$.

\smallskip

Since $u - v$ remains bounded from above over this bounded set, see \eqref{e.bound.uv}, there exists a constant $C < \infty$ such that for every $\al > 0$ sufficiently small, we have
\begin{equation*}  
|t_\al-t_\al'|^2 + |h_\al-h_\al'|^2  \le C \al.
\end{equation*}
After extracting a subsequence if necessary, we may assume that $t_\alpha,t'_\alpha \to t_0$ and $h_\al,h'_\al \to h_0$ as $\alpha \to 0$. Using again \eqref{e.u.diverge.time}, it is clear that $t_0 < T$. 
Since $(t_\al,h_\al,t_\al',h_\al')$ is a maximizer of $\Psi_\al$, we have
\begin{equation}  
\label{e.psi.sup}
\Psi_\al(t_\al,h_\al,t_\al',h_\al') \ge \sup_{[0,\infty)^2} (u-v) \ge u(t_0,h_0) - v(t_0,h_0).
\end{equation}
Since we also have
\begin{equation*}  
\Psi_\alpha(t_\al,h_\al,t_\al',h_\al') \le u(t_\alpha,h_\al) - v(t'_\al,h'_\alpha),
\end{equation*}
and the functions $u$ and $v$ are continuous at $(t_0,h_0)$, we deduce, using \eqref{e.psi.sup} twice, that
\begin{equation}  
\label{e.identif.x0}
\lim_{\al \to 0} \Psi_{\alpha}(t_\al,h_\al,t_\al',h_\al') = u(t_0,h_0) - v(t_0,h_0) = \sup_{[0,\infty)^2} (u-v).
\end{equation}
In view of \eqref{e.max.int} and the second equality in \eqref{e.identif.x0}, we have that $t_0 > 0$, and thus, for every $\al > 0$ sufficiently small, we have that $t_\al  > 0$ and $t'_\al > 0$. 
Note that by definition, the function 
\begin{equation*}  
(t,h) \mapsto u(t,h) - v(t'_\al,h'_\al) -  \frac 1 {2\al} \Ll( |t-t'_\al|^2 + |h-h'_\al|^2 \Rr)
\end{equation*}
has a local maximum at $(t,h) = (t_\al,h_\al)$. If $h_\al = 0$, then the definition of viscosity solutions implies that, for every $\alpha > 0$ sufficiently small,
\begin{equation*}  
\min \Ll(-\frac 1 \alpha (h_\al - h_\al'), \frac{1}{\al} (t_\al - t'_\al) - \frac{2}{\al^2} (h_\al - h'_\al)^2 \Rr)\le -\ep.
\end{equation*}
Since in this case $(h_\al = 0)$ the first term in the minimum is nonnegative, we deduce that
\begin{equation}  
\label{e.first.-ep}
\frac{1}{\al} (t_\al - t'_\al) - \frac{2}{\al^2} (h_\al - h'_\al)^2\le -\ep.
\end{equation}
As can be seen directly, this conclusion also holds when $h_\al > 0$. Similarly, we infer from the fact that the function
\begin{equation*}  
(t',h') \mapsto v(t',h') - u(t_\al,h_\al) + \frac 1 {2\al}\Ll(|t_\al-t'|^2 + |h_\al-h'|^2 \Rr)
\end{equation*}
has a local minimum at $(t',h') = (t'_\al,h'_\al)$ that
\begin{equation*}  
\frac{1}{\al} (t_\al - t'_\al) - \frac{2}{\al^2} (h_\al - h'_\al)^2 \ge \ep.
\end{equation*}
This is in contradiction with \eqref{e.first.-ep}, and thus the proof is complete.
\end{proof}

We now turn to the proof of Proposition~\ref{p.hopf-lax}.

\begin{proof}[Proof of Proposition~\ref{p.hopf-lax}]
We decompose the proof into five steps.

\smallskip

\emph{Step 1.} We show that the function $\psi$ is uniformly Lipschitz. This function is clearly differentiable at every $h > 0$ and
\begin{align*}  
\psi'(h) = \E \la \frac{zx}{2\sqrt{h}} +x\bar x - \frac{x^2}{2}\ra ,
\end{align*}
where here the notation $\la \cdot \ra$ simplifies into 
\begin{equation*}  
\la f(x) \ra := \frac{\int_\R f(x)\exp\Ll(\sqrt{h} z x + h x \bar x - \frac h 2 x^2 \Rr) \, \d P(x)}{\int_\R \exp\Ll(\sqrt{h} z x + h x \bar x - \frac h 2 x^2 \Rr) \, \d P(x)}.
\end{equation*}
By Gaussian integration by parts, we obtain that
\begin{equation}  
\label{e.psi.first.derivative}
\psi'(h) =\E \la x \bar x\ra.
\end{equation}
Since we assume that the support of the measure $P$ is bounded, this completes the proof that $\psi'$ is uniformly bounded.

\smallskip

\emph{Step 2.} For convenience, we  extend $\psi$ to be constant equal to $\psi(0)$ on $(-\infty,0]$, so that for every $t,h \ge 0$,
\begin{align*}  
f(t,h) & = \sup_{h' \in \R} \Ll( \psi({h}') - \frac{({h} - {h}')^2}{8t} \Rr) \\
& = \sup_{h' \in \R} \Ll( \psi(h-{h}') - \frac{(h')^2}{8t} \Rr) .
\end{align*}
For every $h,h_1 \ge 0$, we have
\begin{equation*}  
\psi(h-h') -  \frac{(h')^2}{8t} \le\psi(h_1-h') -  \frac{(h')^2}{8t} + \|\psi'\|_{L^\infty(\R)} \, |h_1 - h|,
\end{equation*}
and thus, taking the supremum over $h'$ on both sides,
\begin{equation*}  
f(t,h) \le f(t,h_1) + \|\psi'\|_{L^\infty(\R)} \, |h_1 - h|.
\end{equation*}
Since the roles of $h$ and $h_1$ are symmetric, this shows that $f$ is uniformly Lipschitz in the $h$ variable.

\smallskip

\emph{Step 3.} As preparation for the proof that $f$ is a viscosity solution of~\eqref{e.hj}, we prove the dynamic programming principle, namely, that for every $t,s,h \ge 0$,
\begin{equation}
\label{e.dynamic.1}
f(t+s,h) = \sup_{h' \ge 0} \Ll( f(t,h') - \frac{(h-h')^2}{8s} \Rr) .
\end{equation}
By convexity of the square function, we have, for every $t,s,h,h',h'' \ge 0$,
\begin{equation}  
\label{e.convex.square}
\Ll( \frac{h-h'}{t+s} \Rr) ^2 \le \frac{t}{t+s} \Ll( \frac{h''-h'}{t} \Rr)^2  + \frac{s}{t+s} \Ll( \frac{h-h''}{s} \Rr)^2,
\end{equation}
and thus
\begin{align*}  
f(t+s,h) & \ge \sup_{h',h'' \ge 0} \Ll( \psi(h') - \frac{(h''-h')^2}{8t} - \frac{(h-h'')^2}{8s} \Rr)  
\\
& \ge \sup_{h'' \ge 0} \Ll( f(t,h'') - \frac{(h-h'')^2}{8s} \Rr) .
\end{align*}
This proves one inequality in \eqref{e.dynamic.1}. The converse inequality is immediate, since the right side of \eqref{e.dynamic.1} is
\begin{equation*}  
\sup_{h',h'' \ge 0} \Ll( \psi(h') - \frac{(h''-h')^2}{8t} - \frac{(h-h'')^2}{8s} \Rr),
\end{equation*}
and, for each fixed $h,h' \ge 0$, we can achieve the case of equality in \eqref{e.convex.square} for some $h'' \ge 0$, so that
\begin{equation*}  
\inf_{h'' \ge 0}\Ll( \frac{(h''-h')^2}{8t} + \frac{(h-h'')^2}{8s} \Rr) =  \frac{(h-h')^2}{8(t+s)}.
\end{equation*}

\smallskip

\emph{Step 4.} We show that $f$ is a viscosity supersolution of \eqref{e.hj}. Let $(t,h) \in (0,\infty)\times [0,\infty)$ and $\phi \in C^\infty((0,\infty)\times [0,\infty)$ be such that $(t,h)$ is a local minimum of $f-\phi$. We start by assuming that $h > 0$. By \eqref{e.dynamic.1}, we have that for every $p \in \R$ and $s > 0$ sufficiently small,
\begin{equation*}  
f(t,h) \ge f(t-s,h-sp) - \frac{sp^2}{8}.
\end{equation*}
Since $f-\phi$ has a local minimum at $(t,h)$, we have that for every $p \in \R$ and $s > 0$ sufficiently small,
\begin{equation*}  
f(t-s,h-sp)-\phi(t-s,h-sp) \ge f(t,h)-\phi(t,h).
\end{equation*}
Combining these two inequalities and passing to the limit $s \to 0$, we deduce that
\begin{equation*}  
\partial_t \phi(t,h) + p \partial_h \phi(t,h) + \frac{p^2}{8} \ge 0,
\end{equation*}
and taking the infimum over $p \in \R$ yields 
\begin{equation}  
\label{e.app.supersol}
\Ll( \partial_t \phi - 2 (\partial_h \phi)^2 \Rr) (t,h) \ge 0,
\end{equation}
as desired. In the case when $h = 0$, the same reasoning applies, except that we need to restrict to values of $p$ such that $p \le 0$. That is, for every $p \le 0$,
\begin{equation*}  
\partial_t \phi(t,h) + p \partial_h \phi(t,h) + \frac{p^2}{8} \ge 0.
\end{equation*}
If $- \partial_h \phi(t,h) \ge 0$, then there is nothing to show. Otherwise, we can choose $p = -4\partial_h \phi(t,h)$ and conclude for the validity of \eqref{e.app.supersol}.

\smallskip

\emph{Step 5.} We show that $f$ is a viscosity subsolution of \eqref{e.hj}.
Let $(t,h) \in (0,\infty)\times [0,\infty)$ and $\phi \in C^\infty((0,\infty)\times [0,\infty)$ be such that $(t,h)$ is a local maximum of $f-\phi$. In view of \eqref{e.dynamic.1} and of the fact that $f$ is uniformly Lipschitz in the $h$ variable, it is clear that for each $s > 0$, there exists $h'_s \in \R$ such that 
\begin{equation*}  
f(t,h) = f(t-s,h-h'_s) - \frac{(h'_s)^2}{8s},
\end{equation*}
and moreover, we have that $h'_s \to 0$ as $s \to 0$. Since $f-\phi$ has a local maximum at $(t,h)$, we have that, for every $s > 0$ sufficiently small,
\begin{equation*}  
(f-\phi)(t-s,h-h_s') \le (f-\phi)(t,h),
\end{equation*}
and thus
\begin{equation}
\label{e.realize.sup}
\phi(t,h) - \phi(t-s,h-h_s') \le - \frac{(h'_s)^2}{8s}.
\end{equation}
We start by assuming that $h > 0$, in which case we aim to show that 
\begin{equation}  
\label{e.app.subsol}
\Ll( \partial_t \phi - 2 (\partial_h \phi)^2 \Rr) (t,h) \le 0.
\end{equation}
We argue by contradiction, assuming the negation of \eqref{e.app.subsol}. By continuity, there exist $\ep, \de > 0$ such that for every $s, a \in [-\de,\de]$, we have
\begin{equation}  
\label{e.app.interior}
\Ll( \partial_t \phi - 2 (\partial_h \phi)^2 \Rr) (t-s,h-a) \ge  \ep,
\end{equation}
and thus, for every such $s$ and $a$ and every $p \in \Rd$,
\begin{equation*}  
\Ll( \partial_t \phi + p \partial_h \phi\Rr) (t-s,h-a) \ge \ep -  \frac{p^2}{8} .
\end{equation*}
By integration, we deduce that, for every $s \in ( 0,\de]$,
\begin{align*}  
\notag
\phi(t,h) - \phi(t-s,h-a) & = \int_0^1 \Ll( s\partial_t \phi + a \partial_h \phi\Rr)(t-\alpha s, h - \alpha a) \, \d \alpha 
\\
\notag
& 
= s \int_0^1 \Ll( \partial_t \phi + \frac{a}{s} \partial_h \phi\Rr)(t-\alpha s, h - \alpha a) \, \d \alpha 
\\
& 
\ge s\ep - \frac{a^2}{8s}.
\end{align*}
For $s > 0$ sufficiently small, we can choose $a = h'_s$ in the inequality above, and reach a contradiction with \eqref{e.realize.sup}. This shows \eqref{e.app.subsol} in the case $h > 0$.

\smallskip

When $h = 0$, our starting point has to be modified from \eqref{e.app.interior} to the statement that for every $s \in [-\de,\de]$ and $a \in [-\delta,0]$,
\begin{equation*}  
\min \Ll( -\partial_h \phi,  \partial_t \phi - 2 (\partial_h \phi)^2 \Rr)(t-s,h-a) \ge \ep.
\end{equation*}
We can then reproduce the argument above and arrive at a contradiction.
\end{proof}

\subsection*{Acknowledgments} I would like to warmly thank Louigi Addario-Berry and Pascal Maillard for stimulating discussions, the CRM and Jessica Lin for hospitality, and Florent Krzakala and Lenka Zdeborov\'a for organizing a conference in  Carg\`ese which inspired the present paper. I was partially supported by the ANR grants LSD (ANR-15-CE40-0020-03) and Malin (ANR-16-CE93-0003) and by a grant from the NYU--PSL Global Alliance.

\small
\bibliographystyle{abbrv}
\bibliography{HJinfer}

\newcommand{\noop}[1]{} \def\cprime{$'$}
\begin{thebibliography}{10}

\bibitem{poincare}
D.~Bakry, F.~Barthe, P.~Cattiaux, and A.~Guillin.
\newblock A simple proof of the {P}oincar\'{e} inequality for a large class of
  probability measures including the log-concave case.
\newblock {\em Electron. Commun. Probab.}, 13:60--66, 2008.

\bibitem{bar1}
J.~Barbier, M.~Dia, N.~Macris, F.~Krzakala, T.~Lesieur, and L.~Zdeborov\'{a}.
\newblock Mutual information for symmetric rank-one matrix estimation: a proof
  of the replica formula.
\newblock In {\em Advances in Neural Information Processing Systems 29}, pages
  424--432, 2016.

\bibitem{bm1}
J.~Barbier and N.~Macris.
\newblock The adaptive interpolation method: a simple scheme to prove replica
  formulas in {B}ayesian inference.
\newblock {\em Probab. Theory Related Fields}, in press.

\bibitem{bmm}
J.~Barbier, N.~Macris, and L.~Miolane.
\newblock The layered structure of tensor estimation and its mutual
  information.
\newblock In {\em 55th Annual Allerton Conference on Communication, Control,
  and Computing}, pages 1056--1063. IEEE, 2017.

\bibitem{barra1}
A.~Barra, G.~Del~Ferraro, and D.~Tantari.
\newblock Mean field spin glasses treated with {PDE} techniques.
\newblock {\em Eur. Phys. J. B}, 86(7):Art. 332, 10, 2013.

\bibitem{barra2}
A.~Barra, A.~Di~Biasio, and F.~Guerra.
\newblock Replica symmetry breaking in mean-field spin glasses through the
  {H}amilton-{J}acobi technique.
\newblock {\em J. Stat. Mech. Theory Exp.}, (9):P09006, 22, 2010.

\bibitem{blm}
S.~Boucheron, G.~Lugosi, and P.~Massart.
\newblock {\em Concentration inequalities}.
\newblock Oxford University Press, Oxford, 2013.

\bibitem{chatterjee}
S.~Chatterjee.
\newblock {\em Superconcentration and related topics}.
\newblock Springer Monographs in Mathematics. Springer, Cham, 2014.

\bibitem{guide}
M.~G. Crandall, H.~Ishii, and P.-L. Lions.
\newblock User's guide to viscosity solutions of second order partial
  differential equations.
\newblock {\em Bull. Amer. Math. Soc. (N.S.)}, 27(1):1--67, 1992.

\bibitem{dm}
Y.~Deshpande and A.~Montanari.
\newblock Information-theoretically optimal sparse {PCA}.
\newblock In {\em IEEE International Symposium on Information Theory}, pages
  2197--2201, 2014.

\bibitem{eak}
A.~{El Alaoui} and F.~Krzakala.
\newblock Estimation in the spiked {W}igner model: a short proof of the replica
  formula, \noop{2018}{preprint, arXiv:1801.01593}.

\bibitem{evans}
L.~C. Evans.
\newblock {\em Partial differential equations}, volume~19 of {\em Graduate
  Studies in Mathematics}.
\newblock American Mathematical Society, Providence, RI, second edition, 2010.

\bibitem{guerra}
F.~Guerra.
\newblock Broken replica symmetry bounds in the mean field spin glass model.
\newblock {\em Comm. Math. Phys.}, 233(1):1--12, 2003.

\bibitem{lm}
M.~Lelarge and L.~Miolane.
\newblock Fundamental limits of symmetric low-rank matrix estimation.
\newblock {\em Probab. Theory Related Fields}, in press.

\bibitem{lkz}
T.~Lesieur, F.~Krzakala, and L.~Zdeborov{\'a}.
\newblock Phase transitions in sparse {PCA}.
\newblock In {\em IEEE International Symposium on Information Theory}, pages
  1635--1639, 2015.

\bibitem{les}
T.~Lesieur, L.~Miolane, M.~Lelarge, F.~Krzakala, and L.~Zdeborov{\'a}.
\newblock Statistical and computational phase transitions in spiked tensor
  estimation.
\newblock In {\em IEEE International Symposium on Information Theory}, pages
  511--515, 2017.

\bibitem{MPV}
M.~M{\'e}zard, G.~Parisi, and M.~Virasoro.
\newblock {\em Spin glass theory and beyond: an introduction to the replica
  method and its applications}, volume~9.
\newblock World Scientific Publishing Company, 1987.

\bibitem{pan}
D.~Panchenko.
\newblock {\em The {S}herrington-{K}irkpatrick model}.
\newblock Springer Monographs in Mathematics. Springer, New York, 2013.

\bibitem{Tpaper}
M.~Talagrand.
\newblock The {P}arisi formula.
\newblock {\em Ann. of Math. (2)}, 163(1):221--263, 2006.

\bibitem{obnoxious}
M.~Talagrand.
\newblock Mean field models for spin glasses: some obnoxious problems.
\newblock In {\em Spin glasses}, volume 1900 of {\em Lecture Notes in Math.},
  pages 63--80. Springer, Berlin, 2007.

\bibitem{Tbook2}
M.~Talagrand.
\newblock {\em Mean field models for spin glasses. {V}olume {II}}, volume~55 of
  {\em Ergebnisse der Mathematik und ihrer Grenzgebiete}.
\newblock Springer, Heidelberg, 2011.

\end{thebibliography}

\end{document}